\documentclass[11pt,reqno]{article}

\usepackage[top=2.5cm, bottom=2.5cm, left=2.5cm, right=2.5cm]{geometry}	

\usepackage{color}
\usepackage{amssymb}
\usepackage{comment}

\usepackage[inline]{showlabels}

\usepackage{amsthm}
\usepackage{amsmath}
\usepackage{paralist}

\usepackage{graphicx}
\usepackage{amsfonts}
\usepackage{amsmath}
\usepackage{amssymb}
\usepackage{amsbsy}
\usepackage{fullpage}
\usepackage{epsfig}
\usepackage{mathrsfs} 
\usepackage{verbatim}
\usepackage[latin1]{inputenc}
\usepackage{enumerate}
\usepackage[hidelinks]{hyperref}

\newtheorem{theorem}{Theorem}
\newtheorem{lemma}{Lemma}
\newtheorem{remark}{Remark}
\newtheorem{Example}{Example}
\newtheorem{assumption}{Assumption}
\newtheorem{prop}{Proposition}

\numberwithin{theorem}{section}
\numberwithin{equation}{section}
\numberwithin{equation}{section}
\numberwithin{prop}{section}
\numberwithin{lemma}{section}
\numberwithin{remark}{section}
\numberwithin{Example}{section}
\numberwithin{assumption}{section}
\newcommand{\hf}{{\hfill$\Box$}}

\newcommand{\kpo}{^{k+1}}

\newcommand{\kn}{^{k,N}_i}
\newcommand{\bn}{^{(N)}}

\newcommand{\ra}{\rightarrow}
\newcommand{\lra}{\longrightarrow}

\newcommand{\kkn}{^{k,N}}

\newcommand{\lv}{\left\vert}
\newcommand{\rv}{\right\vert}

\newcommand{\dd}{,\dots,}
\newcommand{\be}{\begin{equation}}
\newcommand{\ee}{\end{equation}}
\newcommand{\PP}{\mathcal{P}}

\newcommand{\lan}{\langle}
\newcommand{\ran}{\rangle}

\newcommand{\les}{\lesssim}
\newcommand{\hsint}{\mathcal{H}^s_{\cap}}

\newcommand{\tr}{\textup{Trace}}

\newcommand{\N}{\mathbb{N}}

\newcommand{\R}{\mathbb{R}}
\newcommand{\hs}{\mathcal{H}^s}
\newcommand{\h}{\mathcal{H}}
\newcommand{\beq}{\begin{equation}}
\newcommand{\eq}{\end{equation}}
\newcommand{\hl}{h_{\ell}}
\newcommand{\bl}{b_{\ell}}
\newcommand{\al}{\alpha_{\ell}}
\newcommand{\el}{_{\ell}}
\newcommand{\C}{\mathcal{C}}
\newcommand{\EE}{\mathbb{E}}
\newcommand{\cN}{\mathcal{N}}
\newcommand{\dist}{\stackrel{\mathcal{D}}{\sim}}
\newcommand{\exo}{\EE_{x^0}}

\newcommand{\Pb}{\mathbb{P}}
\newcommand{\Eb}{\mathbb{E}}
\newcommand{\Ebc}[1]{\mathbb{E}_k #1}

\newcommand{\Ebb}[1]{\mathbb{E}\left[#1 \right]}
\newcommand{\iprod}[2]{\left\langle #1,#2 \right\rangle}
\newcommand{\iprodcn}[2]{\left\langle #1,#2 \right\rangle_{\C_N}}

\newcommand{\iprods}[2]{\left\langle #1,#2 \right\rangle_{s}}

\newcommand{\norm}[1]{\left|\left|#1\right|\right|}
\newcommand{\normcn}[1]{\left|\left|#1\right|\right|_{\C_N}}
\newcommand{\normc}[1]{\left|\left|#1\right|\right|_{\C}}
\newcommand{\norms}[1]{\left|\left|#1\right|\right|_{s}}
\newcommand{\normms}[1]{\left|\left|#1\right|\right|_{-s}}

\newcommand{\norcN}[1]{\normcn{#1}}
\newcommand{\nors}[1]{\norms{#1}}

\newcommand{\mmag}[1]{\left|#1\right|}

\begin{document}
\title{Non-stationary phase of the MALA algorithm} 
\author{Juan Kuntz\footnote{Department of Mathematics and Department of Bioengineering, Imperial College London, London, SW7 2AZ, UK,              juankuntz@protonmail.com.}        \and
        Michela Ottobre\footnote{Mathematics Department, Heriot Watt University,  Edinburgh, EH14 4AS, UK, michelaottobre@gmail.com.}  \and
        Andrew M. Stuart\footnote{Department of Computing and Mathematical Sciences, California Institute of Technology, CA 91125, USA, astuart@caltech.edu.}
}
\maketitle
\begin{abstract}
The  Metropolis-Adjusted Langevin Algorithm  (MALA) is a Markov Chain Monte Carlo  
(MCMC) method 
which creates a Markov chain reversible with respect to a given 
target distribution, $\pi^N$, with Lebesgue density on $\R^N$; it can hence be used 
to approximately sample the target distribution. When the dimension $N$ is large a key 
question is to determine the computational cost of the algorithm as a 
function of $N$.  A widely adopted approach to this question, which we employ here, 
is to  derive diffusion limits for the algorithm. The scaling of the effective time-step
with respect to $N$ then gives a measure of the number of steps
required by the algorithm. For the
MALA algorithm this problem has been studied when the target measure is of product form,
started both in and out of stationarity. The family of target measures that we consider in this paper are in non-product form and are found from discretization of a measure 
on an infinite dimensional Hilbert space which is defined by its density with respect 
to a Gaussian random field. In this non-product setting the MALA algorithm has been 
studied in stationarity. In extending this work to the non-stationary setting, significant new analysis is required. In particular our diffusion
limit comprises a stochastic PDE, coupled to a scalar ordinary differential equation which is a measure of how far from
stationarity the process is. The results 
demonstrate that, in the non-stationary regime, the number of steps required
by the algorithm is of ${\mathcal O}(N^{1/2})$ 
in contrast to the stationary regime, where 
it is of ${\mathcal O}(N^{1/3})$. For measures defined via density with 
respect to a Gaussian random field, non-standard MCMC methods exist which
require ${\mathcal O}(1)$ steps. However the result proved here about MALA
is of interest because it is consistent with existing results derived for product form 
distributions, and suggests that these results have validity outside the product 
setting. 
\end{abstract}
%
%
%
%
%
%
\section{Introduction}
\subsection{Context}
Metropolis-Hastings algorithms are Markov Chain Monte Carlo (MCMC) methods used to sample from a given probability measure, {referred to as the target measure}. The basic mechanism consists of employing a proposal transition density $q(x, y)$ in order to produce a reversible Markov chain $\{x^k\}_{k=0}^{\infty}$ for which the target measure $\pi$ is 
invariant \cite{Hast:70}. 
At step $k$ of the chain, a proposal move $y^{k}$ is generated by using $q(x,y)$, i.e. 
$y^{k} \sim q(x^k, \cdot)$. Then such a move is accepted with probability $\alpha(x^k, y^k)$:
\be\label{alpha1}
\alpha(x^k,y^k)= \min\left\{1, \frac{\pi(y^k) q(y^k,x^k)}{\pi(x^k) q(x^k,y^k)}  \right\}\,.
\ee

The computational cost of this algorithm when the state space has high dimension $N$ 
is of practical interest in many applications. A widely used approach to this problem is
to study diffusion limits. The inverse of the effective time-step in such analyses
provides an estimate of the computational cost. For product measures this
problem was studied, in the stationary regime, in \cite{MR1428751} for
the random walk Metropolis method \cite{Metr:etal:53} (RWM) and in \cite{Robe:Rose:98}
for the Metroplis adjusted Langevin algorithm \cite{Roberts1996} (MALA). 
RWM was shown to require ${\mathcal O}(N)$ steps and MALA ${\mathcal O}(N^{\frac13}).$
The same ${\mathcal O}(N^{\frac13})$ scaling for MALA, in the stationary regime, was later obtained in the more general settings of target distributions arising from non-linear regression problems \cite{Breyer2004} and that of non-product measures defined via density with respect to a Gaussian random field \cite{MR3024970}. In the paper \cite{MR2137324} extensions of these results to non-stationary
intializations were considered, primarily in the Gaussian product setting. RWM
was shown to scale the same in and out of stationarity, whilst MALA scaled
like ${\mathcal O}(N^{\frac12})$ out of stationarity. In \cite{JLM12MF,JLM12LT}
the RWM and MALA algorithms were studied out of stationarity for quite general
product measures and the RWM method shown to scale the same in and out of stationarity. For MALA the appropriate scaling was shown to differ in and out of
stationarity and, crucially, the scaling out of stationarity was shown to
depend on a certain moment of the potential defining the product measure.
In this paper we contribute further understanding of the MALA algorithm when
intialized out of stationarity by considering non-product measures defined via density with
respect to a Gaussian random field. Doing so has proved fruitful in the study
of both RWM and MALA in stationarity; see \cite{Matt:Pill:Stu:11,MR3024970}.
In this paper our primary contribution is the study of diffusion
limits for the the MALA algorithm,
out of stationarity, in the setting of measures defined via density with
respect to a Gaussian random field. We prove a diffusion limit which
characterizes the computational cost, and is in agreement with the
simple setting of Gaussian product measures considered in \cite{MR2137324}.
Significant new analysis is needed for this problem because the work of
\cite{MR3024970} relies heavilty on stationarity in analyzing the acceptance
probability, whilst the work of \cite{JLM12MF} uses propagation of chaos
techniques, unsuitable for non-product settings.

Despite the challenges in proving the diffusion limit obtained in this paper,
and its relevance to the overall picture just described,
it is important to recognize that, for measures absolutely continuous with respect
to a Gaussian random field, algoritms exist which require ${\mathcal O}(1)$ steps
in and out of stationarity; see \cite{David} for a review. Such methods were suggested by Radford Neal in
\cite{Neal98}, and developed by Alex Beskos for conditioned stochastic differential
equations in \cite{Besk:etal:08}, building on the general formulation of
Metropolis-Hastings methods in \cite{MR1620401}; these methods are analyzed 
from the point of view of diffusion limits in \cite{Pillai2014}. 
It thus remains open and interesting to study
the MALA algorithm out of stationarity for non-product measures which are
not defined via density with respect to a Gaussian random field; however
the results in \cite{JLM12LT} demonstrate the substantial
technical barriers that will exist in trying to do so.
An interesting starting point of such work might be the study of non i.i.d. product
measures as pioneered by B{\'{e}}dard \cite{Beda:07,bedard2008optimal}. Nonetheless, the understanding we obtain about the MALA algorithm when applied to realistic non-product targets provides substantial justification for the analysis that we undertake in this paper.

\subsection{Setting and the Main Result}

Let ($\h, \langle \cdot, \cdot \rangle, \| \cdot \|$) be an infinite dimensional separable Hilbert space and  consider the measure $\pi$ on $\h$, defined as follows:  
\be\label{targetmeasure}
\frac{d\pi}{d\pi_0}  \propto \exp({-\Psi}), \qquad \pi_0:=\cN(0,\C).
\ee
That is, $\pi$ is  absolutely continuous with respect to a Gaussian measure $\pi_0$ with mean zero and covariance operator $\C$. 
 $\Psi$ is some real valued functional with domain $\tilde{\h} \subseteq \h$, $\Psi: \tilde{\h}\ra \R$. { Measures of the form \eqref{targetmeasure}  naturally arise in Bayesian nonparametric statistics and in the study of conditioned diffusions \cite{Stua:10,Hair:etal:05}.} In Section \ref{sec:2} we will give the precise definition of the space $\tilde{\h}$ and identify it with an appropriate Sobolev-like subspace of $\h$ (denoted by $\h^s$ in Section \ref{sec:2}).The covariance operator $\C$ is a positive, self-adjoint, trace class operator on 
$\h$, with eigenbasis $\{\lambda_j^2, \phi_j\} $:
\be\label{cphi}
\C \phi_j= \lambda_j^2 \phi_j, \quad \forall j \in\mathbb{N},
\ee
and we assume that the set $\{\phi_j\}_{j \in \N}$ is an orthonormal basis for $\h$. 

We will analyse the MALA algorithm designed to sample from the finite dimensional projections $\pi^N$
 of the measure \eqref{targetmeasure} {on the space} 
\be\label{XN}
X^N:=\textrm{span}\{\phi_j\}_{j=1}^N \subset \h 
 \ee 
spanned by the first $N$ eigenvectors of the covariance operator. {Notice that the space $X^N$ is isomorphic to $\R^N$.  To clarify this further, we need to introduce some notation. 
Given a point $x \in \h$,  $\PP^N(x):=\sum_{j=1}^n\iprod{\phi_j}{x}\phi_j$ is the projection of $x$ onto the space $X^N$ and we define the approximations of functional $\Psi$ and covariance operator $\C$:
%
\be\label{defpsiNCN}
\Psi^N:=\Psi\circ \PP^N \quad \mbox{and} \quad \C_N:=\PP^N\circ \C\circ \PP^N. \ee
With this notation in place,  our target measure is the measure $\pi^N$ (on $X^N \cong \R^N $) defined as 
\be\label{targetmeasureN}
\frac{d\pi^N}{d\pi_0^N}(x)=M_{\Psi^N}e^{-\Psi^N(x)},  \qquad \pi_0^N:=\cN(0,\C_N), 
\ee
where $M_{\Psi^N}$ is a normalization constant. Notice that the sequence of measures $\{\pi^N\}_{N\in \N}$ approximates the measure $\pi $ (in particular, the sequence $\{\pi^N\}_{N\in \N}$ converges to $\pi$ in the Hellinger metric, see \cite[Section 4]{Stua:10} and references therein). In order to sample from the measure $\pi^N$ in  \eqref{targetmeasureN}, we will consider the MALA algorithm with proposal 
\be\label{proposal}
y^{k,N}=x^{k,N}+\delta \C_N\nabla \log \pi^N(x^{k,N})+ \sqrt{2 \delta}\, \C_N^{1/2} \xi^{k,N}, 
\ee
where
$$
 \xi^{k,N}=\sum_{i=1}^N \xi_i\phi_i,  \quad
\xi_i \stackrel{\mathcal{D}}{\sim} \mathcal{N}(0,1) \mbox{  i.i.d} ,
$$
and $\delta>0$ is a positive parameter. Note that this proposal
may be written as
$$y^{k,N}=x^{k,N}-\delta\bigl(x^{k,N}+ \C_N \nabla \psi^N(x^{k,N})\bigr)+ \sqrt{2 \delta}\, \C_N^{1/2} \xi^{k,N}.$$
 The proposal defines the kernel $q$ and
the accept-reject criteria $\alpha$ which
is added to preserve detailed balance with
respect to $\pi^N$. 
The proposal is a discretization of a $\pi^N$ invariant diffusion process
with time step $\delta$; in the MCMC literature $\delta$ is often referred to as
the proposal variance. The accept-reject criteria compensates for the
discretization which destroys the $\pi^N$ reversibility.
A crucial parameter to be appropriately chosen in order to optimize the performance 
of the algorithm is $\delta$: it is desirable to determine the largest possible
$\delta$, as a function of dimension $N$, which leads to an order one acceptance
probability, as a function of $N$. The inverse of this $\delta$ gives the number
of steps required by the algorithm, as a function of $N$. 

We now come to explain the main result of the paper.  We show that if
\begin{equation}
\label{eq:pv}
\delta=\ell/\sqrt{N}
\end{equation}
then the acceptance probability is ${\mathcal O}(1)$. 
Furthermore, starting from the {Metropolis-Hastings chain $\{x^{k,N}\}_{k\in\N}$}, we
define the continuous interpolant
\be\label{interpolant}
x\bn (t)=(N^{1/2}t-k)x^{k+1,N}+(k+1-N^{1/2}t)x^{k,N}, \quad t_k\leq t< t_{k+1},  
\mbox{ where } t_k=\frac{k}{N^{1/2}}.
\ee
This process converges weakly to a diffusion process.
The precise statement of such a  result is given in Theorem \ref{thm:mainthm1} (and  Section \ref{Sec:sec5} contains heuristic arguments which explain how such a result is obtained).  
In proving the result we will use the fact that
$W(t)$ is a  $\h_s$-valued Brownian motion with covariance $\C_s$ with $\h_s$ a 
(Hilbert) subspace of $\h$ and $\C_s$ the covariance in this space. Details of these
spaces are given in Section \ref{sec:2}, in particular \eqref{BMCS} and \eqref{glue}.
Below  $C([0,T];\h_s)$ denotes the space of $\h_s$-valued continuous functions on $[0,T]$, endowed with the uniform topology;  $\alpha\el, h\el$ and $b\el$ are real valued functions, which we will define immediately after the statement, and $x^{k,N}_j$ denotes the $j$-th component of the vector {$x^{k,N}\in X^N$ with respect to the basis $\{\phi_1,\dots,\phi_N\}$} (more details on this notation are given in Subsection \ref{subsec:notation}.)

\smallskip

{\bf Main Result.} {\em  
 Let $\{x^{k,N}\}_{k\in\N}$ be the Metropolis-Hastings Markov chain {to sample from $\pi^N$ and } constructed using the MALA proposal \eqref{proposal} (i.e. the chain \eqref{chainxcomponents}) with $\delta$ chosen to satisfy \eqref{eq:pv}. 
{Then, for any deterministic initial datum
 $x^{0,N}=\PP^N(x^0)$, where   $x^0$ is any point in  $\h_s$,} the continuous interpolant 
$x\bn$ defined in \eqref{interpolant} converges weakly in $C([0,T];\h_s)$  to the  solution of the SDE
\be\label{informalSPDe}
dx(t)=- h\el (S(t)) \bigl(x(t)+\C\nabla\Psi(x(t)) \bigr)  \, dt+\sqrt{2h\el (S(t))}  \, dW(t)  , \quad x(0)=x^0, 
\ee
where $S(t) \in \R_+:=\{s\in \R: s\geq0\}$ solves the ODE
\be\label{ODE}
dS(t)=b_{\ell}(S(t))\, dt, \qquad S(0):= \lim_{N \ra \infty}
 \frac{1}{N}\sum_{j=1}^N \frac{\lv x_j^{0,N} \rv^2}{\lambda_j^2} .
 \ee
In the above the initial datum $S(0)$ is assumed to be finite and $W(t)$ is a 
$\h_s$-valued Brownian motion with covariance $\C_s$. 
}
\smallskip

The functions $\alpha\el, h\el, b\el: \R \rightarrow \R$ in the previous statement are defined as follows:
\begin{align}
\alpha\el(s) &=  1\wedge e^{\ell^2 (s-1)/2}  \label{def:al}\\
h\el (s) &= \ell \alpha\el(s) \label{def:hl}\\
b\el(s)& =  2\ell (1-s)\left( 1\wedge e^{\ell^2 (s-1)/2} \right) = 2 (1-s) h\el(s). \label{defbl}
\end{align}

\begin{remark}\label{rem:onsindepofx}
We make several remarks concerning the main result.
\begin{itemize}
\item  Since the effective time-step implied by the interpolation \eqref{interpolant}
is $N^{-1/2}$, the main result implies that the number of steps required by the
Markov chain in its non-stationary regime is ${\mathcal O}(N^{1/2})$.  A more
detailed discussion  on this fact can be found in Section \ref{Sec:sec5}.

\item Notice that equation \eqref{ODE} evolves independently of equation \eqref{informalSPDe}.  
Once the MALA algorithm \eqref{chainxcomponents} is introduced and an initial state 
$x^0\in \tilde{\h}$ is given such that $S(0)$ is finite, the  real valued (double) sequence $S^{k,N}$, 
\be\label{skn}
S^{k,N}:=\frac{1}{N} \sum_{i=1}^N \frac{\lv x^{k,N}_i\rv^2}{\lambda_i^2}
\ee
started at 
$S_0^N:=\frac{1}{N} \sum_{i=1}^N \frac{\lv x^{0,N}_i\rv^2}{\lambda_i^2}$ is well defined. For fixed $N$,  $\{S^{k,N}\}_k$ is not, in general, a Markov process (however it is Markov if e.g. $\Psi=0$).   Consider the continuous interpolant $S\bn (t)$ of the sequence $S^{k,N}$, namely
\be\label{interpolantofsk}
S\bn (t)=(N^{1/2}t-k)S^{k+1,N}+(k+1-N^{1/2}t)S^{k,N}, 
\quad t_k\leq t< t_{k+1},  \,\, 
 t_k=\frac{k}{N^{\frac12}}.
\ee
In Theorem \ref{thm:weak conv of Skn} we  prove that $S\bn(t)$ converges in probability in $C([0,T];\R)$ to the solution of the ODE \eqref{ODE} with initial condition 
$S_0:=\lim_{N\ra\infty}S_0^N$. Once such a result is obtained, we can prove that $x\bn(t)$ converges to $x(t)$. We want to stress that the convergence of 
$S\bn(t)$ to $S(t)$ can be obtained independently of the convergence of $x\bn(t)$ to $x(t)$. 
\item  
Let $S(t):\R\rightarrow \R$ be the solution of the ODE \eqref{ODE}. We will prove (see Theorem \ref{thm:existenceuniquenessforODE}) that   $S(t) \ra 1$  as $t\ra \infty$. With this in mind,   notice that
$h\el(1)=\ell$.  Heuristically one can then argue that the asymptotic behaviour of the law of $x(t)$, 
the solution of \eqref{informalSPDe}, 
is described by the law of the following infinite dimensional SDE:
\be\label{longlimSPDEinf}
dz(t)=-\ell (z(t)+\C \nabla \Psi(z(t)))dt+ \sqrt{2\ell} dW(t). 
\ee
It was proved in \cite{Hair:etal:05,Hair:Stua:Voss:07}  that \eqref{longlimSPDEinf} is ergodic with unique invariant measure given by  \eqref{targetmeasure}. 
Our deduction concerning computational cost is made on the assumption
that the law of \eqref{informalSPDe} does indeed tend to the law of
\eqref{longlimSPDEinf}, although we will not prove this here as it would take
us away from the main goal
of the paper which is to establish the diffusion limit of the MALA algorithm.  

\item In \cite{JLM12MF,JLM12LT} the diffusion limit for the MALA algorithm started out of stationarity and applied to i.i.d. target product measures is given by  a non-linear equation of McKean-Vlasov type. This is in contrast with our diffusion limit, which is  an infinite-dimensional SDE. The reason why this is the case is discussed in detail in \cite[Section 1.2]{KOS16}. The discussion in the latter paper  is in the context of
the Random Walk Metropolis algorithm, but it is conceptually analogous to what holds 
for the MALA algorithm and for this reason we do not spell it out here.  

\item In this paper we make stronger assumptions on $\Psi$ than are required 
to prove a diffusion limit in the stationary regime \cite{MR3024970}. In
particular we assume that the first deriveative of $\Psi$ is bounded, whereas
\cite{MR3024970} requires only boundedness of the second derivative. Removing
this assumption on the first derivative, or showing that it is necessary, 
would be of interest but would require different techniques to those employed in this paper and we do not address the issue here. 

\end{itemize}
\end{remark}

\subsection{Structure of the paper}
The paper is organized as follows. In  Section  \ref{sec:2} we introduce the notation and the assumptions  that we use throughout this paper. In particular, Subsection \ref{subsec:notation} introduces the infinite dimensional setting that we work in and Subsection \ref{ssec:algor} discusses the MALA algorithm and the assumptions we make on the functional $\Psi$ and on the covariance operator $\mathcal{C}$. Section \ref{sec:sec4} contains the proof of existence and uniqueness of solutions for the limiting  equations \eqref{informalSPDe} and \eqref{ODE}. With these preliminaries  in place, we give in Section \ref{Sec:sec5}, the formal statement of the main results of this paper,  {Theorems \ref{thm:weak conv of Skn} and \ref{thm:mainthm1}}. In this section we also provide  heuristic arguments outlining how the main results are obtained. The complete proof of these results builds on a continuous mapping argument presented in Section \ref{sec6}. The heuristics of Section \ref{Sec:sec5} are made rigorous in {Sections \ref{sec7}--\ref{sec9}}. In particular, Section \ref{sec7} establishes some estimates of the size of the chain's jumps and the growth of its moments and certain approximations of the acceptance probability. In Sections \ref{sec8} and \ref{sec9} we use these estimates and approximations to prove Theorem \ref{thm:weak conv of Skn} and Theorem    \ref{thm:mainthm1}, respectively. Readers interested in the structure of the proofs of Theorems \ref{thm:weak conv of Skn} and \ref{thm:mainthm1} but not in the technical details may wish to skip the ensuing two sections (Sections \ref{sec:2} and \ref{sec:sec4}) and proceed directly to the statement of these results and the relevant heuristics discussed in Section \ref{Sec:sec5}.
\section{Notation, Algorithm, and Assumptions}\label{sec:2}
In this section we detail the notation and the assumptions (Section \ref{subsec:notation}
and Section  \ref{sec:assumptions}, respectively) that we will use in the rest of the paper. 
\subsection{Notation}\label{subsec:notation}
Let $\left( \h, \langle\cdot, \cdot \rangle, \|\cdot\|\right)$ 
denote a {real} separable infinite dimensional Hilbert space, with
the canonical norm induced by the inner-product.  Let $\pi_0$ be a zero-mean Gaussian measure on $\h$ with covariance operator $\C$. By the general theory of Gaussian measures \cite{PZ92}, $\C$ is a positive, trace class operator. Let $\{\phi_j,\lambda^2_j\}_{j \geq 1}$ be the eigenfunctions
and eigenvalues of $\C$, respectively, so that
 \eqref{cphi} holds.
We assume a normalization under which $\{\phi_j\}_{j \geq 1}$ 
forms a complete orthonormal basis {of} $\h$. Recalling \eqref{XN}, we specify the notation that will be used throughout this paper: 
\begin{itemize}
\item $x$ and $y$ are elements of the Hilbert space $\h$;
\item the letter $N$ is reserved to denote the  dimensionality of the space $X^N$ where
the target measure $\pi^N$ is supported;
\item $x^N$ is an element of $X^N$ { $\cong \R^N$} (similarly for {$y^N$ and the noise $\xi^N$});
\item {for any fixed $N \in \N$,}  $x^{k,N}$ is the $k$-th step of the chain $\{x^{k,N}\}_{k \in \N} \subseteq  X^N$ constructed to sample from $\pi^N$; $x^{k,N}_i$ is the $i$-th component of the vector $x^{k,N}$, that is $x^{k,N}_i:=\langle x^{k,N}, \phi_i\rangle$ {(with abuse of notation)}.
\end{itemize}
For every $x \in \h$, we have the representation
$x = \sum_{j\geq1} \; x_j \phi_j$, where $x_j:=\langle x,\phi_j\rangle.$ Using this expansion, we define Sobolev-like spaces $\h^s, s \in \R$, with the inner-products and norms defined by
$$
\langle x,y \rangle_s = \sum_{j=1}^\infty j^{2s}x_jy_j
\qquad \text{and} \qquad
\|x\|^2_s = \sum_{j=1}^\infty j^{2s} \, x_j^{2}.
$$
The space $(\h^s, \langle \cdot, \cdot \rangle_s)$ is also a Hilbert space. {Notice that $\h^0 = \h$.} Furthermore
$\h^s \subset \h \subset \h^{-s}$ for any $s >0$.  
The Hilbert-Schmidt norm $\|\cdot\|_\C$ {associated with the covariance operator $\C$} is defined as
$$
\normc{x}^2 := \sum_{j=1}^{\infty} \lambda_j^{-2} x_j^2= \sum_{j=1}^{\infty} \frac{\lv \lan x, \phi_j\ran\rv^2}{\lambda_j^2},\qquad x\in\h, 
$$
{and it is the Cameron-Martin norm associated with the Gaussian measure $\cN(0,\C)$.}
Such a norm is induced by the scalar product
$$
\lan x, y\ran_{\C} :=\lan \C^{-1/2}x, \C^{-1/2}y \ran,  \qquad x,y\in\h.
$$
Similarly, {$\C_N$  defines a Hilbert-Schmidt norm on $X^N$,
\be\label{norcn}
\norcN{x^N}^2:=\sum_{j=1}^{N} \frac{\lv \lan x^N, \phi_j\ran\rv^2}{\lambda_j^2},\qquad x^N\in X^N,
\ee}
which is induced by the scalar product 
$$
\lan x^N, y^N\ran_{\C_N} :=\lan \C_N^{-1/2}x^N, \C_N^{-1/2}y^N \ran, \qquad x^N,y^N\in X^N.
$$
For $s \in \mathbb{R}$, 
let  $L_s : \h \rightarrow \h$ denote the operator which is
diagonal in the basis $\{\phi_j\}_{j \geq 1}$ with diagonal entries
$j^{2s}$, 
$$
L_s \,\phi_j = j^{2s} \phi_j,
$$
so that $L^{\frac12}_s \,\phi_j = j^s \phi_j$. 
The operator $L_s$ 
lets us alternate between the Hilbert space $\h$ and the interpolation spaces $\h^s$ via the identities:
$$
\langle x,y \rangle_s = \langle L^{\frac12}_s x,L^{\frac12}_s y \rangle 
\qquad \text{and} \qquad
\|x\|^2_s =\|L^{\frac12}_s x\|^2. 
$$
Since $\norms{L_s^{-1/2} \phi_k} = \norm{\phi_k}=1$, 
we deduce that $\{\hat{\phi}_k:=L^{-1/2}_s \phi_k \}_{k \geq 1}$ forms an 
orthonormal basis {of} $\h^s$. An element $y\sim \cN(0,\C)$ can be  expressed as
\be\label{y1}
y=\sum_{j=1}^{\infty}
\lambda_j \rho_j \phi_j  \qquad \mbox{with } \qquad \rho_j\stackrel{\mathcal{D}}{\sim}\cN(0,1) \,\,\mbox{i.i.d}.
\ee
If $\sum_j \lambda_j^2 j^{2s}<\infty$, then $y$ can be equivalently  written as
\be\label{y2}
y=\sum_{j=1}^{\infty}
(\lambda_j j^s) \rho_j (L_s^{-1/2} \phi_j)  \qquad \mbox{with } \qquad \rho_j\stackrel{\mathcal{D}}{\sim}\cN(0,1) \,\,\mbox{i.i.d}.
\ee
For a positive, self-adjoint operator $D : \h \mapsto \h$, its trace in $\h$ is defined as
$$
\tr_{\h}(D) \;{:=}\; \sum_{j=1}^\infty \langle \phi_j, D  \phi_j \rangle.
$$
We stress that in the above $ \{ \phi_j \}_{j \in \mathbb{N}} $ is an orthonormal basis for $(\h, \langle \cdot, \cdot \rangle)$. Therefore, if 
$\tilde{D}:\h^s \rightarrow \h^s$, its trace in $\h^s$ is
$$
\tr_{\h^s}(\tilde{D}) \;{=}\; \sum_{j=1}^\infty \langle L_s^{-\frac{1}{2}} \phi_j, \tilde{D} L_s^{-\frac{1}{2}} \phi_j \rangle_s.
$$
Since $\tr_{\h^s}(\tilde{D})$ does not depend on the orthonormal basis,
the operator $\tilde{D}$ is said to be trace class in $\h^s$ if $\tr_{\h^s}(\tilde{D}) < \infty$ for
some, and hence any, orthonormal basis of $\h^s$. 
Because $\C$ is defined on $\h$, the covariance operator\footnote{In this paper, we commit a slight abuse of our notation by writing $\C_s$ to mean the covariance operator on the Sobolev-like subspace $\hs$ and $\C_N$ to mean that on the finite dimensional subspace $X^N$ as defined in \eqref{defpsiNCN}. We distinguish these two by always employing $N$ as the subscript for the latter, and lower case letters such as $s$ or $r$ for the former.}
\be\label{glue}
\C_s=L_s^{1/2} \C L_s^{1/2}
\ee
is defined on $\h^s$.  Thus, 
for all the values of $r$ such that $\tr_{\h^s}(\C_s)=\sum_j \lambda_j^2 j^{2s}< \infty$, we can think of $y$ as a mean zero Gaussian random variable with covariance operator $\C$ in $\h$ and $\C_s$ in $\h^s$ (see \eqref{y1} and \eqref{y2}). {In the same way, 
if $\tr_{\h^s}(\C_s)< \infty$, then 
\be\label{BMCS}
W(t)= \sum_{j=1}^{\infty} \lambda_j w_j(t) \phi_j= \sum_{j=1}^{\infty}\lambda_j j^r w_j(t) \hat{\phi}_j,
\ee
where $\{ w_j(t)\}_{j \geq1}$  a collection of  i.i.d. standard Brownian motions on $\mathbb{R}$, 
can be equivalently  understood as an $\h$-valued $\C$-Brownian motion or as an  $\h^s$-valued $\C_s$-Brownian motion.  } 

We will  make use of the following elementary inequality,
\begin{equation}\label{eq:B2}
\mmag{\iprod{x}{y}}^2=\mmag{\sum_{j=1}^{\infty}
(j^s x_j)(j^{-s}y_j)}^2 \leq \norms{x}^2   \normms{y}^2,\qquad \forall x \in \h^s,\quad y \in \h^{-s} \,.
\end{equation}
\noindent
{Throughout this paper we study sequences of real numbers, random variables and functions, indexed by either (or both) the dimension $N$ of the space on which the { target measure} is defined or the chain's step number $k$. In doing so, we find the following notation convenient.
\begin{itemize}
\item Two (double) sequences of real numbers $\{A^{k,N}\}$ and $\{B^{k,N}\}$ satisfy $A^{k,N} \lesssim B^{k,N}$ 
if there exists a constant $K>0$ (independent of $N$ and $k$) such that
$$A^{k,N}\leq KB^{k,N},$$
for all $N$ and $k$ such that $\{A^{k,N}\}$ and $\{B^{k,N}\}$ are defined. 
\item If the $A^{k,N}$s and $B^{k,N}$s are random variables, the above inequality must hold almost surely (for some deterministic constant $K$).
\item If the $A^{k,N}$s and $B^{k,N}$s are real-valued functions on $\h$ or $\h^s$, $A^{k,N}= A^{k,N}(x)$ and $B^{k,N}= B^{k,N}(x)$,  the same inequality must hold with $K$ independent of $x$, for all $x$ where the $A^{k,N}$s and $B^{k,N}$s are defined. 
\end{itemize}}
As is customary,  $\R_+:=\{s\in \R : s \geq 0\}$ and for all $b \in \R_+$ we let $[b]=n$ if $n\leq b < n+1$ for some integer $n$. Finally,  for time dependent functions we will  use both the notations $S(t)$ and $S_t$ interchangeably.  

\subsection{The Algorithm}\label{ssec:algor}

A natural variant of the
MALA algorithm stems from the observation that $\pi^N$ is the unique stationary measure of the SDE
\begin{equation}\label{eq:LSDE}dY_t=\C_N\nabla\log\pi^N(Y_t)dt+\sqrt{2}dW^N_t,\end{equation}
where $W^N$ is an {$X^N$-valued}  Brownian motion with covariance operator $\C_N$. The algorithm consists of discretising \eqref{eq:LSDE} using the Euler-Maruyama scheme and adding a Metropolis accept-reject step so that the invariance of $\pi^N$ is preserved. 
The variant on MALA which we study is therefore a Metropolis-Hastings algorithm with proposal  
\be\label{eqn:proposal}
y^{k,N} =x^{k,N}- \delta\left (  x^{k,N}+ \C_N \nabla \Psi^N(x^{k,N})\right)+
\sqrt{2\delta} \C_N^{1/2} \xi^{k,N},  
\ee
where
$$
\xi^{k,N}:= \sum_{j=1}^N \xi^{k,N}_j \phi_j, \quad  \xi^{k,N}_j \sim  \cN(0,1){\mbox{ i.i.d}}.
$$
{We stress that the Gaussian random variables $\xi^{k,N}_i$ are independent of each other and of the current position $x^{k,N}$.} 
Motivated by  the considerations made in the introduction (and that will be made more explicit in Section \ref{sec5.1}), in this paper we fix the choice
\be\label{delta}
\delta:=\frac{\ell}{N^{1/2}}. 
\ee
If at step $k$ the chain is at $x^{k,N}$, the algorithm proposes a move to $y^{k,N}$ defined by equation \eqref{eqn:proposal}. The move is then accepted with probability 
\be\label{accprob1}
\alpha^N(x^{k,N},y^{k,N}):=\frac{\pi^N(y^{k,N}) q^N(y^{k,N}, x^{k,N})}{\pi^N(x^{k,N}) q^N(x^{k,N}, y^{k,N})},
\ee
where, for any $x^N, y^N \in \R^N \simeq  X^N$, 
\be\label{qq}
q^N(x^N,y^N)\propto e^{-\frac{1}{4\delta}\|(y^N-x^N)-\delta \nabla\log \pi^N(x^N)\|^2_{\C_N}}.
\ee
If the move to $y^{k,N}$ is accepted then $x^{k+1,N}=y^{k,N}$, if it is rejected the chain remains where it was, i.e. $x^{k+1,N}=x^{k,N}$.  In short, the MALA chain is defined as follows:
\be\label{chain-gamma}
x^{k+1,N}:=\gamma^{k,N} y^{k,N}+ (1-\gamma^{k,N})x^{k,N},\qquad x^{0,N}:=\PP^N(x^0)
\ee
where in the above 
\be\label{defgammaaccept}
\gamma^{k,N}\dist \textup{Bernoulli}( \alpha^N(x^{k,N},y^{k,N}));
\ee
that is, conditioned on $(x^{k,N},y^{k,N})$, $\gamma^{k,N}$ has Bernoulli law with mean $\alpha^N(x^{k,N},y^{k,N})$. 
{Equivalently, we can write
 $$
 \gamma^{k,N}={\bf{1}}_{\left\{U^{k,N}\leq \alpha^N(x^{k,N},y^{k,N})\right\}},
 $$
  with $U^{k,N}\dist $\,Uniform$\,[0,1]$, independent of $x^{k,N}$ and $\xi^{k,N}$.}

For fixed $N$, the chain $\{x^{k,N}\}_{k\geq1}$ lives in $X^N \cong \R^N$ and samples from $\pi^N$. However, in view of the fact that we want to study the scaling limit of such a chain as $N \ra \infty$, the analysis is cleaner if it is carried out in $\h$; therefore,  the chain that we analyse is the chain $\{x^k\}_{k}\subseteq \h$ defined as follows:   the first $N$ components of the vector $x^k \in \h$ coincide with $x^{k,N}$ as defined above; the remaining components are not updated and remain equal to their initial value.    More precisely, using \eqref{eqn:proposal} and \eqref{chain-gamma},  the chain $x^k$   can be written in a component-wise    notation as follows:  
\begin{align}
x^{k+1}_i=x^{k+1,N}_i &=x\kn -  \gamma^{k,N} \left[\frac{\ell}{N^{1/2}}\left (  x\kn+  [\C_N \nabla \Psi^N(x^{k,N})]_i\right)+
\sqrt{\frac{2 \ell}{N^{1/2}}} \lambda_i \,\xi^{k,N} \right]\,   
\label{chainxcomponents}
\end{align}
for $i=1 \dd N$, while 
\begin{align}
x\kpo &=x^k=x^0 \qquad \mbox{on } \h \setminus X^N.  \nonumber 
 \end{align}
For the sake of clarity, we specify that  $[\C_N \nabla \Psi^N(x\kkn)]_i$ denotes the $i$-th component of the vector $\C_N \nabla \Psi^N(x\kkn) \in \hs$. From the above it is clear that the update rule \eqref{chainxcomponents} only updates the first $N$ coordinates (with respect to the eigenbasis of $\C$) of the vector $x^k$.  Therefore the algorithm evolves in the finite-dimensional subspace $X^N$.  From now on we will avoid using the notation $\{x^k\}_k$ for the ``extended chain" defined in $\h$, as it can be confused with the notation $x^N$, which instead is used throughout to  denote a generic element of the space $X^N$.

We conclude this section by remarking that, if $x^{k,N}$ is given,  the proposal $y^{k,N}$ only depends on the Gaussian noise  $\xi^{k,N}$. Therefore the acceptance probability will be interchangeably denoted by $\alpha^N(x^N,y^N)$ or $\alpha^N(x^N,\xi^N)$. 
\subsection{Assumptions} \label{sec:assumptions}
In this section we describe the assumptions on the covariance 
operator $\C$ of the Gaussian measure $\pi_0 \dist \cN(0,\C)$ and those on the  functional $\Psi$. We fix a distinguished exponent 
$s\geq 0$ and assume that $\Psi: \mathcal{H}^s\rightarrow \R$ and $\tr_{\mathcal{H}^s}(\mathcal{C}_s)<\infty$.  In other words $\h^s$ is the space that we were denoting with $\tilde{\h}$ in the introduction. 
Since  
\be\label{ljf}
\tr_{\mathcal{H}^s}(\mathcal{C}_s)= \sum_{j=1}^{\infty} \lambda_j^2 j^{2s},
\ee
the condition $\tr_{\mathcal{H}^s}(\mathcal{C}_s)<\infty$ implies that  
$\lambda_j j^s \ra 0$ as $j \ra \infty$. Therefore the sequence 
 $\{\lambda_j j^s\}_j$ is bounded:
 \be\label{bddseq}
 \lambda_j j^s \leq C,
 \ee
 for some constant $C>0$ independent of $j$. 
 
For each $x \in \h^s$ the derivative $\nabla \Psi(x)$
is an element of the dual $\mathcal{L}(\h^s,\R)$ of $\h^s$, comprising 
the linear functionals on $\h^s$.
However, we may identify $ \mathcal{L}(\h^s,\R)=\h^{-s}$ and view $\nabla \Psi(x)$
as an element of $\h^{-s}$ for each $x \in \h^s$. With this identification,
the following identity holds
\begin{equation}
\norm{ \nabla \Psi(x)}_{\mathcal{L}(\h^s,\R)} = \normms{\nabla \Psi(x)}.
\end{equation}
%
To avoid technicalities we assume that the gradient of $\Psi(x)$  is bounded { and globally Lipschitz.}  More precisely,  throughout this paper we make the following assumptions.
\begin{assumption} \label{ass:1}
The functional $\Psi$ and covariance operator $\C$ satisfy the following:
\begin{enumerate}
\item {\bf Decay of Eigenvalues $\lambda_j^2$ of $\C$:}
there exists a constant $\kappa > \frac{1}{2}$ such that
$$
\lambda_j \asymp j^{-\kappa}.
$$
\item {\bf Domain of $\Psi$:}
there exists an exponent $s \in [0, \kappa - 1/2)$ such 
that $\Psi$ is defined everywhere on $\h^s$.
\item {\bf Derivatives of $\Psi$:} 
The derivative of $\Psi$ is bounded and globally Lipschitz:
\begin{equation}\label{eq:C2}
\normms{ \nabla \Psi(x)} \lesssim 1,\qquad\normms{ \nabla \Psi(x)- \nabla \Psi(y)} \lesssim\norms{x-y}.
\end{equation}

\end{enumerate}
\end{assumption}
\begin{remark} \label{rem:one}\textup{
The condition $\kappa > \frac{1}{2}$ ensures that $\tr_{\h^s}(\C_s)  < \infty$
for any $0 \leq s  < \kappa - \frac{1}{2}${. Consequently, $\pi_0$ has support in $\hs$ ($\pi_0(\h^s)=1$) 
for any  $ 0 \leq s < \kappa - \frac{1}{2}$.  }}  {\hfill$\Box$}
\end{remark}

\begin{Example}\label{exder}\textup{
The functional $\Psi(x)  = \sqrt{1+\norms{x}^2}$ satisfies all of the above.}  {\hfill$\Box$} 
\end{Example}

\begin{remark}\textup{Our assumptions on the change of measure (that is, on $\Psi$) are less general than those adopted in \cite{KOS16, MR3024970} and related literature (see references therein). This is for purely technical reasons. In this paper we assume that $\Psi$ grows linearly. If $\Psi$ was assumed to grow quadratically, which is the case in the mentioned works,  finding bounds on the moments of the  chain {$\{x^{k,N}\}_{k\geq 1}$}   (much needed in all of the analysis) would become more involved than it already is, see Remark \ref{Remmomchain}. However, under our assumptions, the measure $\pi$ (or $\pi^N$) is still,  generically, of non-product form. }  \hf \end{remark}
\noindent We now explore the consequences of  Assumption \ref{ass:1}. The proofs of the following lemmas can be found in Appendix \ref{misc}.
%
\begin{lemma} \label{lem:lipschitz+taylor}
Suppose that Assumption  \ref{ass:1} holds. Then 
\begin{enumerate}
\item  The function $\C \nabla \Psi(x)$ is bounded and  globally Lipschitz on $\hs$, that is
\begin{equation}\label{eq:lipz2}
\norms{\C\nabla\Psi(x)}\les 1 \quad \mbox{and} \quad \norms{\C\nabla\Psi(x)-\C\nabla\Psi(y)}\lesssim \norms{x-y}   \,. 
\end{equation}
Therefore,  
the function  $F(z):=-z-\C \nabla \Psi(z)$ satisfies
\begin{equation}\label{e.Flipshitz}
\norms{F(x) - F(y)} \lesssim  \norms{x-y} \quad \mbox{and}
\quad \norms{F(x)} \lesssim  1+ \norms{x}\,  \,. 
\end{equation}
\item The function $\Psi(x)$  is globally Lipschitz and therefore also $\Psi^N(x):=\Psi(\PP^N(x))$ is globally Lipschitz:
\begin{equation}\label{eq:taylor}
\mmag{\Psi^N(y)-\Psi^N(x)}\lesssim\norms{y-x}\,. 
\end{equation}
\end{enumerate}
\end{lemma}
%
Before stating the next lemma, we observe that by  definition of the projection operator $\PP^N$ we have that
\be\label{gradpsiN}
 \nabla \Psi^N=\PP^N\circ \nabla \Psi\circ \PP^N.
 \ee
\begin{lemma}\label{lemma2.6}
Suppose that Assumption  \ref{ass:1} holds.  Then the following holds for the  function $\Psi^N$  and for its the gradient:
\begin{enumerate}
\item If the bounds \eqref{eq:C2} hold for $\Psi$, then they hold for $\Psi^N$ as well:
\begin{equation}\label{eq:lin}
\normms{\nabla\Psi^N(x)}\lesssim 1,\qquad\normms{ \nabla \Psi^N(x)- \nabla \Psi^N(y)} \lesssim\norms{x-y}.
\end{equation}
\item {Moreover}, 
\begin{equation}\label{eq:lipz}
\norm{\C_N\nabla \Psi^N(x)}_s\lesssim 1,
\end{equation}
and 
\begin{equation}\label{eq:B4}
\normcn{\C_N\nabla\Psi^N(x)}\lesssim 1.
\end{equation}
\end{enumerate}
\end{lemma}
We stress that in \eqref{eq:lin}-\eqref{eq:B4} the constant implied by the use of the notation ``$ \les $" (see end of Section \ref{subsec:notation}) is independent of $N$. Lastly, in what follows we will need the fact that, due assumptions on the covariance operator, 
\be\label{c1/2xi}
\EE\nors{\C_N^{1/2} \xi^N}^2 \lesssim 1, \quad \mbox{uniformly in $N$}, 
\ee
where $\xi^N:=\sum_{j=1}^N\xi_j\phi_j$ and $\xi_i \stackrel{\mathcal{D}}{\sim} \mathcal{N}(0,1)$ i.d.d., see \cite[(2.32)]{Matt:Pill:Stu:11} { or \cite[first proof of Appendix A]{KOS16} }
%
%
%
\section{Existence and Uniqueness for the Limiting Diffusion Process}\label{sec:sec4}
The main results of this section are Theorem \ref{thm:existenceuniquenessforODE}, Theorem \ref{Thm:SPDe} and Theorem  \ref{contofupsilon}. Theorem \ref{thm:existenceuniquenessforODE} and  Theorem \ref{Thm:SPDe} are concerned with establishing existence and uniqueness for equations  \eqref{informalSPDe} and \eqref{ODE},
 respectively.  Theorem  \ref{contofupsilon} states the continuity of the It\^o maps associated with equations \eqref{informalSPDe} and \eqref{ODE}. The proofs of the main results of this paper (Theorem \ref{thm:weak conv of Skn} and Theorem \ref{thm:mainthm1}) rely heavily on the continuity of such maps, as we illustrate in Section \ref{sec6}. 

Once Lemma \ref{lem:propofDandGamma} below is established, the proofs of the theorems in this section are completely analogous to the proofs of those in \cite[Section 4]{KOS16}. For this reason, we omit them and refer the reader to \cite{KOS16} for the details. In particular, for the proof of Theorem \ref{thm:existenceuniquenessforODE} see that of \cite[Theorem 4.1]{KOS16}, for that of Theorem \ref{Thm:SPDe} see that of \cite[Theorem 4.3]{KOS16}, and for that of Theorem \ref{contofupsilon} see that of \cite[Theorem 4.6]{KOS16}. 
\begin{lemma}\label{lem:propofDandGamma}
The functions $\al(s)$,  $\hl(s)$ and  $\sqrt{\hl (s)}$ are positive, globally Lipschitz continuous and bounded.  The function $\bl(s)$  is globally Lipschitz and it is  bounded above but not below. 
Moreover, for any $\ell>0$,  $\bl(s)$ is strictly positive for $s\in[0,1)$, strictly negative for $s>1$ and $\bl(1)=0$. 
\end{lemma}
\begin{proof}
When $s>1$, $\al(s)=1$ while for $s\leq 1$ $\al(s)$ has bounded derivative; therefore $\al(s)$ is globally Lipshitz. A similar reasoning gives the Lipshitzianity of the other functions. The further properties of $\bl$ are straightforward from the definition.
\end{proof} 
\begin{theorem}\label{thm:existenceuniquenessforODE}
For any initial datum $S(0) \in \R_+$,  there exists a unique  solution  $S(t) \in \R$ to the ODE \eqref{ODE}. Such a solution  is  strictly positive for any $t>0$,  it is   bounded and has continuous first derivative for all $t\geq 0$.  
 In particular                                   
$$
\lim_{t\ra \infty} S(t) =1 \,
$$
and 
\be\label{solofODEisbdd}
0\leq \min\{S(0),1\}\leq S(t) \leq \max\{S(0), 1\} \, .
\ee
\end{theorem}
We recall that the definition of the functions $\al, h\el$ and $b\el$ has been given in \eqref{def:al}, \eqref{def:hl} and \eqref{defbl}, respectively.
We now come to existence and uniqueness for equation 
\eqref{informalSPDe}, which we rewrite using the notation of 
Lemma \ref{lem:lipschitz+taylor} as 
$$
dx(t)=- h\el (S(t)) F(x(t)) \, dt+\sqrt{2h\el (S(t))}  \, dW(t) ,
$$
where $W(t)$ is an $\hs$-valued 
$\C_s$-Brownian motion. The above is intended to mean 
\be\label{xsolofSPDE}
x(t)=x(0)+\int_0^t F(x(v))h_{\ell}(S(v)) dv+ \int_0^t \sqrt{2h\el (S(v))}  dW(v) 
\, .
\ee 
\begin{theorem}\label{Thm:SPDe}
Let Assumption \ref{ass:1} hold and consider equation \eqref{informalSPDe}(or, equivalently, equation \eqref{xsolofSPDE}), where $W(t)$ is any $\h^s$-valued ${{\C}}_s$-Brownian motion and $S(t)$ is  the solution of  \eqref{ODE}. Then for any initial condition
$ x(0)\in \hs$ and  any $T>0$  there exists a
unique solution of  equation  \eqref{informalSPDe}    in the space
$C([0,T]; \h^s)$.
\end{theorem}
%
%
%
%
Consider now the following    equation:
\be\label{SPDe1}
dx(t)=[-x(t)-\C\nabla\Psi(x(t))]  \hl (S(t)) \, dt + d\zeta(t),
\ee
where $S(t)$ is the solution of \eqref{ODE} and  $\zeta(t)$ is any function in   $C([0,T];\hs)$. Also, let 
$\mathfrak{S}(t):\R_+ \ra \R$ be the solution of 
\be\label{SPDe2}
d\mathfrak{S}(t)=b_{\ell}(\mathfrak{S}(t)) \, dt+ a\, dw(t),
\ee
where  $w(t)$ is a real valued standard Brownian motion and $a\in \R_+$ is a constant. Also, throughout the paper the spaces $C([0,T];\hs)$ and $C([0,T];\R)$ are assumed to be endowed with the uniform topology.
 \begin{remark}\label{rems:deceqn}\textup{ Before stating the next theorem we need to be more precise  about equations \eqref{SPDe1} and \eqref{SPDe2}. 
\begin{itemize}  
\item We consider equation \eqref{SPDe2} (which is \eqref{ODE} perturbed by noise)  in view of the contraction mapping argument  (explained in Section \ref{sec6}) that we will use to prove our main results.  Observe that  \eqref{SPDe2}  admits a unique solution, thanks to the Lipschitzianity of $b\el$. Existence and uniqueness of the solution of \eqref{SPDe1} can be done with identical arguments to those used to prove existence and uniqueness of the solution to \eqref{informalSPDe}.
\item We emphasize that \eqref{SPDe1} and  \eqref{SPDe2} are decoupled as the function $S(t)$ appearing in 
\eqref{SPDe1} is the solution of \eqref{ODE}. This fact will be  particularly  relevant in the remainder of this section as well as in  Section \ref{sebs:contmaparg2} and Section \ref{subs:cma2}. 
\end{itemize}
 } {\hfill$\Box$}
\end{remark}
The  statement of the following theorem is  crucial to the proof of our main result.
\begin{theorem}\label{contofupsilon} 
With the notation introduced so far (and in particular with the clarifications of Remark \ref{rems:deceqn}) 
let $x(t)$ and $\mathfrak{S}(t)$ be the solutions of  \eqref{SPDe1} and \eqref{SPDe2},  respectively.
  Then, under Assumption \ref{ass:1}, the It\^o maps
\begin{align*}
\mathcal{J}_1: \hs  \times  C([0,T]; \hs)  & \lra C([0,T];\hs \times \R) \\
 (x_0,\zeta(t))  & \lra x(t)
\end{align*}
and 
\begin{align*}
\mathcal{J}_2: \R_+ \times  C([0,T]; \R)  & \lra C([0,T]; \R) \\
 (\mathfrak{S}_0, w(t))  & \lra \mathfrak{S}(t)
\end{align*}
are  continuous maps. 
\end{theorem}
%
%
\section{Main Theorems and  Heuristics of proofs}\label{Sec:sec5}

In order to state the main  results, we first 
 set
\be\label{spaceint}
\hsint:=\left\{x \in \hs : \lim_{N \ra \infty} \frac{1}{N}\sum_{i=1}^N  \frac{\lv x_i \rv^2} { \lambda_i^2}< \infty \right\}\,,
\ee
where we recall that in the above $x_i:= \iprod{x}{\phi_i}$. 

 \begin{theorem}\label{thm:weak conv of Skn}
{Let Assumption \ref{ass:1}  hold and let $\delta=\ell/N^{\frac12}$. 
Let $x^0\in \hsint$ and $T>0$.} Then, as $N\to\infty$,  the continuous interpolant $S\bn(t)$ of the {sequence $\{S^{k,N}\}_{k\in\N} 
\subseteq \R_+$ } (defined in  \eqref{interpolantofsk})  {and started at $S^{0,N}=\frac{1}{N}\sum_{i=1}^N  \lv x_{i}^{0} \rv^2 / \lambda_i^2 $},   converges in probability  in $C([0,T]; \R)$ to the  solution $S(t)$ of the ODE \eqref{ODE} with initial datum  $S^0:=\lim_{N\ra \infty}S^{0,N}$.
\end{theorem}

For the following theorem recall that the solution of \eqref{informalSPDe} is
interpreted precisely through Theorem \ref{Thm:SPDe} 
as a process driven by an $\h^s-$valued
Brownian motion with covariance $\C_s$, and solution in $C([0,T];\h^s).$

\begin{theorem}\label{thm:mainthm1}
Let Assumption \ref{ass:1} hold let $\delta=\ell/N^{\frac12}$.
Let  $x^0\in \hsint$ and $T>0$. Then, as $N \ra \infty$,  the continuous interpolant $x\bn(t)$ of the chain {$\{x^{k,N}\}_{k\in\N} \subseteq \hs$} (defined in \eqref{interpolant} and \eqref{chainxcomponents}, respectively) with initial state $x^{0,N}:=\PP^N(x^0)$, converges weakly in $C([0,T]; \hs)$ to the  solution $x(t)$ of equation \eqref{informalSPDe} with initial datum $x^0$. We recall that the time-dependent function $S(t)$ appearing in \eqref{informalSPDe} is the solution of the ODE \eqref{ODE}, started at $S(0):= \lim_{N \ra \infty} \frac{1}{N}\sum_{i=1}^N  \lv x_i^{0} \rv^2 / \lambda_i^2$.
\end{theorem}

Both Theorem \ref{thm:weak conv of Skn} and Theorem \ref{thm:mainthm1} assume that the initial datum of the chains $x^{k,N}$ is assigned deterministically. {From our proofs it will be clear that the same statements also hold for} random initial data, as long as i) $x^{0,N}$ is not drawn at random from the target measure $\pi^N$ or from any other measure which is a change of measure from $\pi^N$ (i.e. we need to be starting out of stationarity) and ii)  $S^{0,N}$ and $x^{0,N}$ have bounded moments (bounded uniformly in $N$) of sufficiently high order and are independent of all the other sources of noise present in the algorithm. Notice moreover that the convergence in probability of Theorem \ref{thm:weak conv of Skn} is equivalent to weak convergence, as the limit is deterministic. 

The rigorous proof of the above results is contained in Sections \ref{sec6} to  \ref{sec9}. In the remainder of this section we give heuristic arguments to justify our choice of scaling $\delta \propto N^{-1/2}$ and we explain how one can formally obtain the (fluid) ODE  limit \eqref{ODE} for the double sequence $S^{k,N}$ and  the diffusion limit  \eqref{informalSPDe} for the chain $x^{k,N}$.
We stress that the arguments of this section are only formal; therefore, we often use the notation $``\simeq"$, to mean ``approximately equal". That is, we write  $A\simeq B$ when $A=B+$ ``terms that are negligible" as $N$ tends to infinity; we then justify these approximations,
and the resulting limit theorems, in the following Sections \ref{sec6} to  \ref{sec9}.

%


\subsection{Heuristic analysis of the acceptance probability}\label{sec5.1}
As observed in \cite[equation (2.21)]{MR3024970},  the acceptance probability \eqref{accprob1} can be expressed as
\be\label{accprobQQ}
\alpha^N(x^N,\xi^N)= 1\wedge e^{Q^N(x^N,\xi^N)},
\ee
where, using the notation \eqref{norcn}, the function $Q^N(x,\xi)$ can be written as
\begin{align}
Q^N(x^N, \xi^N) &  :=  - \frac{\delta}{4} \left(  \norcN{y^N}^2 - \norcN{x^N}^2\right)+ r^N (x^N, \xi^N)  \label{QN2}\\
& = \left[\frac{\delta^2}{2}\left(  \norcN{x^N}^2- \norcN{\C_N^{1/2}\xi^N}^2 \right)\right]- 
\frac{\delta^3}{4} \norcN{x^N}^2  \nonumber\\
&-\left( \frac{\delta^{3/2}}{\sqrt{2}}  -  \frac{\delta^{5/2}}{\sqrt{2}}   \right)    \lan x^N, \C_N^{1/2} \xi^N\ran_{\C_N}+r_{\Psi}^N (x^N, \xi^N) \,.\label{QN1}
\end{align}
We do not give here a complete expression for the terms $r^{N}(x^N,\xi^N)$ and $r^N_{\Psi}(x^N,\xi^N)$. For the time being it is sufficient to point out  that 
\begin{align}
r^N(x^N,\xi^N) &:=I_2^N+ I_3^N \nonumber\\
r^N_{\Psi}(x^N,\xi^N) & := r^N(x^N,\xi^N) + \frac{\left(\delta^2-\delta^3 \right) }{2} \langle x^N, \C_N \nabla \Psi^N(x^N)\rangle_{\C_N}\nonumber\\
& - \frac{\delta^3}{4}\| \C_N \nabla \Psi^N(x^N) \|_{\C_N}^2+ \frac{\delta^{5/2}}{\sqrt{2}}
\langle \C_N \nabla \Psi^N(x^N), \C_N^{1/2}\xi^N \rangle_{\C_N}\label{diff}
\end{align}
where $I_2^N$ and $I_3^N$ will be defined in \eqref{A2} and \eqref{A3}, respectively. Because  $I_2^N$ and $I_3^N$ depend on  $\Psi$, $r^N_{\Psi}$ contains all the terms where the functional $\Psi$ appears; moreover $r^N_{\Psi}$ vanishes when $\Psi=0$.  The analysis of Section \ref{sec7}  (see Lemma \ref{lem:lem:I1}) will show that with our choice of scaling, $\delta= \ell/ N^{1/2}$, the terms $r^N$ and  $r^N_{\Psi}$ are negligible (for $N$ large).
 Let us now illustrate the reason behind our  choice of scaling. To this end,  set $\delta = \ell/ N^{\zeta}$ and observe the following two simple facts:
\be\label{defskn}
S^{k,N}= \frac{1}{N}\sum_{j=1}^N \frac{\lv x^{k,N}_j\rv^2}{\lambda_j^2}= \frac{1}{N} \norcN{x^{k,N}}^2 
\ee
and
\be\label{lawlargen}
\norcN{\C_N^{1/2}\xi^N}^2=\sum_{i=1}^N \lv \xi_i\rv^2 \simeq N,
\ee
the latter fact being true by the Law of Large Numbers. 
 Neglecting the terms containing $\Psi$,  at step $k$ of the chain we have, formally, 
\begin{align}
Q^N(x^{k,N}, \xi^{k+1,N}) & \simeq  \frac{\ell^2}{2} N^{1-2\zeta} \left( S^{k,N}- 1 \right) \label{qnappr1}\\
&  - \frac{\ell^3}{4} N^{1-3\zeta} S^{k,N} - \frac{\ell^{3/2}}{\sqrt{2}} N^{(1-3\zeta)/2} \frac{\lan x^{k,N}, \C_N^{1/2} \xi^{k,N}\ran_{\C_N}}{\sqrt{N}} \label{qnappr2}\\
& - \frac{\ell^{5/2}}{\sqrt{2}} N^{(1-5\zeta)/2} \frac{\lan x^{k,N}, \C_N^{1/2} \xi^{k,N}\ran_{\C_N}}{\sqrt{N}}. \label{qnappr3}
\end{align}
 The above approximation (which, we stress again, is only formal and will be made rigorous in subsequent sections) has been obtained  from  \eqref{QN1} by setting $\delta = \ell/ N^{\zeta}$ and using \eqref{defskn} and \eqref{lawlargen}, as follows:

 \begin{align}
 \frac{\delta^2}{2} \left[\norcN{x^N}^2- \norcN{\C_N^{1/2}\xi^N}^2\right]  & \simeq \eqref{qnappr1}, \label{bla}\\
 - \delta^3 \frac{\norcN{x^N}^2}{4}
-  \frac{\delta^{3/2}}{\sqrt{2}} \lan x^N, \C_N^{1/2} \xi^N\ran_{\C^N}  & \simeq \eqref{qnappr2}, \nonumber\\
- \frac{\delta^{5/2}}{\sqrt{2}} \lan x^N, \C_N^{1/2} \xi^N\ran_{\C^N} & = \eqref{qnappr3}\,. \nonumber
 \end{align}
Looking at the  decomposition \eqref{qnappr1}-\eqref{qnappr3} of the function $Q^N$, we can now heuristically explain  the reason why we are lead to choose $\zeta=1/2$ when we start the chain out of stationarity, as opposed to the scaling $\zeta=1/3$ when the chain is started in stationarity. This is explained in the following remark. 
\begin{remark}\label{rem:scal}\textup{
First notice that the expression \eqref{QN1}  and the approximation \eqref{qnappr1}-\eqref{qnappr3} for $Q^N$ are  valid both in and out of stationarity, as the first is only a consequence of the definition of the Metropolis-Hastings algorithm and the latter is implied just by the properties of $\Psi$ and by our definitions. 
\begin{itemize}
\item  If we start the chain in stationarity, i.e. $x_0^N\sim \pi^N$ (where $\pi^N$ has been defined in \eqref{targetmeasureN}), then $x^{k,N} \sim \pi^N$ for every $k \geq 0$. As we have already observed,    $\pi^N$ is  absolutely continuous with respect to the  Gaussian measure $\pi_0^N \sim \cN(0, \C_N)$; because all the almost sure properties are preserved under this change of measure, in the stationary regime most of the estimates of interest need to be shown only for $x^N \sim \pi_0^N$.  In particular if   $x^N \sim \pi_0^N$ then $x^N$ can be represented as $x^N= \sum_{i=1}^N \lambda_i \rho_i \phi_i$, where $\rho_i$ are i.i.d.~ $\cN(0,1)$. Therefore we can use the law of large numbers and observe that $\|x^N\|_{\C^N}^2=\sum_{i=1}^N \lv\rho_{i} \rv^2 \simeq N $. 
\item Suppose we want to study the algorithm in stationarity and we therefore make the choice $\zeta=1/3$. With the above point in mind, notice  that if we start in stationarity then by the Law of Large numbers $N^{-1}\sum_{i=1}^N \lv\rho_{i} \rv^2= S^{k,N}\rightarrow1$ (as $N\rightarrow \infty$, with speed of convergence $N^{-1/2}$). Moreover, if $x^N \sim \pi_0^N$, by the Central Limit Theorem  the term $\lan x^N, \C_N^{1/2} \xi^N\ran_{\C_N}/\sqrt{N}$ is $O(1)$ and converges to a standard Gaussian. With these two observations in place we can then heuristically see that, with the choice $\zeta=1/3$  the term in    \eqref{qnappr3} are negligible as  $N\ra \infty$ while the terms in \eqref{qnappr2} are $O(1)$. The term in  \eqref{qnappr1} can be better understood by looking at the LHS of \eqref{bla} which, with $\zeta=1/3$ and $x^N \sim \pi_0^N$,  can be rewritten as 
\be\label{help}
\frac{\ell^2}{2N^{2/3}} \sum_{i=1}^N (\lv \rho_i\rv^2- \lv \xi_i\rv^2 ).
\ee
The expected value of the above expression is zero. If we apply   the Central Limit Theorem to the i.i.d. sequence $\{\lv \rho_i\rv^2- \lv \xi_i\rv^2 \}_i$, \eqref{help} shows that \eqref{qnappr1} is $O(N^{1/2-2/3})$ and therefore negligible as $N \ra \infty$. 
 In conclusion,  in the stationary case the only $O(1)$ terms are those in \eqref{qnappr2}; therefore one has the heuristic approximation
$$
Q^N(x,\xi) \sim \mathcal{N} \left( -\frac{\ell^3}{4}, \frac{\ell^3}{2}\right) \,. 
$$
For more details on the stationary case see \cite{MR3024970}. 
\item  If instead we start out of stationarity the choice $\zeta=1/3$ is problematic. Indeed in \cite[Lemma 3]{MR2137324} the authors study the MALA algorithm to sample from an $N$-dimensional isotropic Gaussian and show that if the algorithm is started at a point $x^0$ such that  $S(0) <1$, then  the acceptance probability degenerates to zero. Therefore,  the algorithm stays stuck in its initial state and never proceeds to the next move, see \cite[Figure 2]{MR2137324} (to be more precise, as $N$ increases the algorithm will take longer and longer to get unstuck from its initial state; in the limit, it will never move with probability 1). Therefore the choice $\zeta=1/3$ cannot be the optimal one (at least not irrespective of the initial state of the chain) if we start out of stationarity.  This is still the case in our context and one can heuristically  see that the root of the problem lies in the term \eqref{qnappr1}. Indeed if out of stationarity we still choose $\zeta=1/3$ then, like before, \eqref{qnappr2} is still order one and \eqref{qnappr3} is still negligible. However, looking at \eqref{qnappr1}, if $x^0$ is such that $S(0)<1$ then, when $k=0$, \eqref{qnappr1} tends to minus infinity;  recalling \eqref{accprobQQ}, this implies that the acceptance probability of the first move tends to zero.     To overcome this issue and make $Q^N$ of order one (irrespective of the initial datum) so that the acceptance probability is of order one and does not degenerate to $0$ or $1$ when $N \ra \infty$,   we take  $\zeta=1/2$; in this way the terms in \eqref{qnappr1} are $O(1)$, all the others are small. 
Therefore, the intuition leading the analysis of the non-stationary regime hinges on the fact that, with our scaling, 
\be\label{apprQNskN}
Q^N(x^{k,N}, \xi^{k,N}) \simeq \frac{\ell^2}{2}(S^{k,N} -1);
\ee
hence 
\be\label{star}
\alpha^N(x^{k,N}, \xi^{k,N}) =  (1 \wedge e^{Q^N(x^{k,N}, \xi^{k,N})}) \simeq \alpha\el (S^{k,N}),
\ee
where the function $\alpha_{\ell}$ on the RHS of \eqref{star} is the one defined in \eqref{def:al}. 
The approximation \eqref{apprQNskN}  is made rigorous in Lemma \ref{lem:lem:I1}, while \eqref{star}  is formalized in Section \ref{subsaccprop} (see in particular Proposition \ref{propac}).
\item Finally, we mention for completeness that, by arguing similarly to what we have done so far, if $\zeta < 1/2$ then the acceptance probability of the first move tends to zero when $S(0)<1$. 
If $\zeta >1/2$ then $Q^N \rightarrow 0$, so the acceptance probability tends to one; however the size of the moves is small and the algorithm explores the phase space slowly.
\end{itemize}
}
\end{remark}
\begin{remark}\label{rem:asympindep}\textup{
Notice that in stationarity the function $Q^N$ is, to leading order,  independent of $\xi$; that is, $Q^N$ and $\xi$ are asymptotically independent (see \cite[Lemma 4.5]{{MR3024970}}). This can be intuitively explained  because in stationarity the leading order term in the expression for $Q^N$ is the term with $\delta^3 \|x\|^2$. We will show that also out of stationarity $Q^N$ and $\xi$ are asymptotically independent. In this case such an asymptotic independence can, roughly speaking,  be  motivated by the approximation \eqref{apprQNskN}, (as  the interpolation of the chain $S^{k,N}$ converges to a deterministic limit). The asymptotic correlation of $Q^N$ and the noise $\xi$ is analysed in Lemma \ref{lem:epsilon}. }
\end{remark}

 \subsection{Heuristic derivation of the weak limit of  $S^{k,N}$}\label{secheurlimskn}
 Let $Y$ be any function of the random variables $\xi^{k,N}$ and $U^{k,N}$ (introduced in subsection \ref{ssec:algor}), for example the chain $x^{k,N}$ itself. Here and throughout the paper we use $\exo\left[Y\right]$ to denote the expected value of $Y$ with respect to the law of  the variables $\xi^{k,N}$'s and $U^{k,N}$'s, with the initial  state $x_0$ of the chain given deterministically; in other words, $\exo(Y)$ denotes expectation with respect to all the sources of randomness present in $Y$. We will use the notation $\Ebc\left[Y\right]$ for the conditional expectation of $Y$ given $x^{k,N}$,  $\Ebc\left[Y\right]:=\exo\left[Y\left|x^{k,N}\right.\right]$ (we should really be writing $\Eb_k^N$ in place of $\EE_k$, but to improve readability we will omit the further index $N$).
 Let us now decompose the chain $S^{k,N}$ into its drift and martingale part:
\be\label{dmdecskn}
S^{k+1,N}=S^{k,N}+\frac{1}{\sqrt{N}} \bl^{k,N}+ \frac{1}{N^{1/4}}M^{k,N}, 
\ee
where 
\be\label{belln}
\bl^{k,N}:=\sqrt{N}\Ebc[S^{k+1,N}-S^{k,N}]
\ee
and 
\be\label{adsk11111}
M\kkn:= N^{1/4}\left[ S^{k+1,N}-S^{k,N} - \frac{1}{\sqrt{N}}\bl^{k,N}(x^{k,N})\right] \,.
\ee

In this subsection we give the heuristics which underly the proof,
given in subsequent sections, that the approximate drift $\bl^{k,N}= \bl^{k,N}(x^{k,N})$ converges to $b_{\ell}(S^{k,N})$, \footnote{Notice that $S^{k,N}$ is only a function of $x^{k,N}$} where $b_{\ell}$ is the drift of \eqref{ODE}, while the approximate diffusion $M^{k,N}$ tends to zero. This formally gives the result of Theorem  \ref{thm:weak conv of Skn}. Let us formally argue such a convergence result.   
By \eqref{defskn} and \eqref{chain-gamma}, 
\be\label{C}
S^{k+1,N}= \frac{1}{N} \sum_{j=1}^N \frac{\lv x^{k+1,N}_j\rv^2}{\lambda_j^2}
= \frac{1}{N} \left(  \gamma\kkn \norcN{y\kkn}^2 + (1-\gamma\kkn )\norcN{x^{k,N}}^2 \right) \,.
\ee
Therefore, again by \eqref{defskn}, 
\begin{align}
\bl^{k,N}=\sqrt{N} \EE_k [S^{k+1,N} -S^{k,N}]&= \frac{1}{\sqrt{N}} \EE_k \left[ \gamma\kkn (\norcN{y\kkn}^2 -\norcN{x^{k,N}}^2 ) \right]\nonumber \\
& = \frac{1}{\sqrt{N}} \EE_k \left[  (1 \wedge e^{Q^N(x^{k,N},y^{k,N})}) (\norcN{y\kkn}^2 -\norcN{x^{k,N}}^2 )   \right], \label{starstar}
\end{align}
where the second equality is a consequence of the definition of $\gamma\kkn$ (with a  reasoning, completely analogous to the one in   \cite[last proof of Appendix A]{KOS16}, see also \eqref{ups}). 
Using \eqref{QN2} (with $\delta=\ell/\sqrt{N}$), the fact that $r^N$ is negligible  and  the approximation \eqref{apprQNskN}, the above gives
$$
b_{\ell}^{k,N}=\sqrt{N} \EE_k [S^{k+1,N} -S^{k,N}] \simeq - \frac{4}{\ell}
\left(  1 \wedge e^{\ell^2 (S^{k,N}-1)/2} \right) \frac{\ell^2}{2} (S^{k,N}-1) = b\el(S^{k,N}) \,.
$$
The above approximation is made rigorous in Lemma \ref{lem:driftskn}.
As for the diffusion coefficient, it is easy to check (see proof of Lemma \ref{lem:noiseskn}) that 
$$
N \EE_k [S^{k+1,N} -S^{k,N}]^2 <\infty.
$$
Hence the approximate diffusion tends to zero and one can formally deduce that (the interpolant of) $S\kkn$ converges to the ODE limit \eqref{ODE}.
\subsection{Heuristic analysis of the limit of the chain $x^{k,N}$. }\label{heuranchainx}
 The drift-martingale decomposition of the chain  $x^{k,N}$  is as follows:
\be\label{driftmartdecompx}
 x^{k+1,N}=x^{k,N}+\frac{1}{N^{1/2}}\Theta^{k,N}+\frac{1}{N^{1/4}}L\kkn
\ee
where $\Theta^{k,N}=\Theta^{k,N}(x^{k,N})$  is the approximate drift
\be\label{approximatedriftd}
\Theta^{k,N}:=\sqrt{N} \Ebc\left[ x^{k+1,N}-x^{k,N}\right]
\ee
and 
\be\label{Mkn}
L\kkn:=N^{1/4} \left[x^{k+1,N}- x^{k,N} - \frac{1}{\sqrt{N}} \Theta^{k,N}(x^{k,N})  \right]
\ee
is the approximate diffusion. In what follows we will use the notation $\Theta(x,S)$ for  the drift of equation \eqref{informalSPDe}, i.e.
\be\label{Theta}
\Theta(x, S)= F(x)\hl(S), \quad (x, S) \in \hs \times \R,
\ee
with $F(x)$ defined in Lemma \ref{lem:lipschitz+taylor}. Again, we want to formally argue that the approximate drift $\Theta^{k,N}(x^{k,N})$ tends to $\Theta(x^{k,N}, S^{k,N})$ \footnote {Note that in the limit the dependence of the drift on $S^{k,N}$ becomes explicit.}and the approximate diffusion $L^{k,N}$ tends to the diffusion coefficient of equation \eqref{informalSPDe}. 

\subsubsection{Approximate drift.}
As a preliminary consideration, observe that 
\be\label{ups}
\EE_k \left( \gamma^{k,N}\C_N^{1/2} \xi^{k,N}\right)= \EE_k \left(\left(1 \wedge e^{Q^N(x^{k,N}, \xi^{k,N})} \right) \C_N^{1/2} \xi^{k,N}\right),
\ee
see \cite[equation (5.14)]{KOS16}. This fact will be used throughout the paper, often without mention. 
Coming to the chain $x^{k,N}$, a direct calculation based on \eqref{eqn:proposal} and on \eqref{chain-gamma} gives
\be\label{gracchio}
x^{k+1,N} - x^{k,N} = - \gamma\kkn \delta (x^{k,N} + \C_N \nabla \Psi^N(x^{k,N})) + 
\gamma\kkn \sqrt{2 \delta} \C_N^{1/2} \xi\kkn. 
\ee
Therefore, with the choice $\delta = \ell/\sqrt{N}$, we have
\begin{align}
\Theta^{k,N}=\sqrt{N}\EE_k [x^{k+1,N} -x^{k,N}] &=  -\ell \EE_k \left[ (1 \wedge e^{Q^N(x^{k,N},\xi^{k,N})}) (x^{k,N}+ \C_N\nabla\Psi^N(x^{k,N}))  \right] \nonumber\\
&+{N^{1/4}} \sqrt{2\ell} \EE_k \left[ (1 \wedge e^{Q^N(x^{k,N},\xi^{k,N})})  \C_N^{1/2}  \,  \xi^{k,N}\right]  \label{bit2apprdrift}
\end{align}
The addend in \eqref{bit2apprdrift} is asymptotically small (see Lemma \ref{lem:epsilon} and notice that this addend would just be zero if $Q^N$ and $\xi^{k,N}$ were uncorrelated); hence, using the heuristic approximations \eqref{apprQNskN} and \eqref{star}, 
\begin{align}
\Theta^{k,N}=\sqrt{N}\EE_k [x^{k+1,N} -x^{k,N}]& \simeq - \ell \alpha\el (S^{k,N} )(x^{k,N}+ \C_N\nabla\Psi^N(x^{k,N}))\nonumber\\
&\stackrel{\eqref{def:hl}}{=} - h_{\ell}(S^{k,N}) (x^{k,N}+ \C_N\nabla\Psi^N(x^{k,N}))\label{blu};
\end{align}
the right hand side of the above is precisely  the limiting drift $\Theta(x^{k,N},S^{k,N})$.

\subsubsection{Approximate diffusion.}\label{subs:addx}
We now look at  the approximate diffusion of the chain $x\kkn$:
$$
L^{k,N}:= N^{1/4} (x^{k+1,N}-x^{k,N}-\EE_k(x^{k+1,N}-x^{k,N}) ).
$$
By definition, 
\begin{align}\label{LK}
\EE_k\nors{L^{k,N}}^2&= \sqrt{N}\EE_k \nors{x^{k+1,N}-x^{k,N}}^2 
- \sqrt{N}\nors{\EE_k\left( x^{k+1,N}-x^{k,N}\right)}^2.
\end{align}
By \eqref{blu} the second addend in the above is asymptotically small. Therefore
\begin{align*}
\EE_k\nors{L^{k,N}}^2& \simeq \sqrt{N}\EE_k \nors{x^{k+1,N}-x^{k,N}}^2\\
&\stackrel{\eqref{chain-gamma}, \eqref{gracchio}}{\simeq} {2\ell} \EE_k \nors{\gamma\kkn \C_N^{1/2} \xi\kkn}^2\\
&= {2\ell}\EE_k\sum_{j=1}^N j^{2s}\lambda_j^2\left( 1 \wedge e^{Q^N(x^{k,N},\xi^{k,N})} \right) \lv \xi\kkn_j\rv^2.
\end{align*}
The above quantity is carefully studied in  Lemma \ref{lem:AI}. However, intuitively,  the heuristic approximation \eqref{star} (and the asymptotic independence of $Q^N$ and $\xi$ that \eqref{star} is a manifestation of) suffices to formally derive the limiting diffusion coefficient (i.e. the diffusion coefficient of \eqref{informalSPDe}): 
\begin{align*}
\EE_k\nors{L^{k,N}}^2& \simeq 2\ell  \sum_{j=1}^N j^{2s}\lambda_j^2 \EE_k
\left[(1 \wedge e^{Q^N(x^{k,N},y^{k,N})} )\lv \xi_j^{k,N}\rv^2\right]\\
& \simeq 2\ell  \sum_{j=1}^N j^{2s}\lambda_j^2 \EE_k
\left[(1 \wedge e^{\ell^2(S^{k,N}-1)/2} )\lv \xi_j^{k,N}\rv^2\right] \simeq 
2\ell  \sum_{j=1}^N j^{2s}\lambda_j^2 (1 \wedge e^{\ell^2(S^{k,N}-1)/2} )\\
& \simeq 2\ell \, \tr(\C_s)\alpha\el (S^{k,N})\stackrel{\eqref{def:hl}}{=}2\tr(\C_s)\,h\el (S^{k,N}). 
\end{align*}

\section{Continuous Mapping Argument}\label{sec6}

{In this section we outline the argument which underlies the proofs of our main results. In particular,   the proofs of Theorem \ref{thm:weak conv of Skn} and Theorem \ref{thm:mainthm1} hinge on the continuous mapping arguments that we illustrate in the following Section \ref{sebs:contmaparg2} and Section \ref{subs:cma2}, respectively.    The details of the proofs are deferred to  the next three sections: Section \ref{sec7} contains some preliminary results that we employ in both proofs, Section \ref{sec8} contains the the proof of Theorem \ref{thm:weak conv of Skn} and Section \ref{sec9} that of Theorem \ref{thm:mainthm1}.}

\subsection{Continuous Mapping Argument for \eqref{SPDe2}}\label{sebs:contmaparg2}
Let us recall the definition of the chain $\{S^{k,N}\}_{k\in\N}$  and of its  continuous interpolant $S^{(N)}$, introduced in \eqref{skn} and \eqref{interpolantofsk}, respectively.   {From the definition \eqref{interpolantofsk} of the interpolated process  and the drift-martingale decomposition \eqref{dmdecskn} of the chain $\{S^{k,N}\}_{k\in\N}$ we have that for any $t \in [t_k, t_{k+1})$,
\begin{align}
S^{(N)}(t) &= (N^{1/2}t-k) \left[ S^{k,N}+\frac{1}{\sqrt{N}} \bl^{k,N}+ \frac{1}{N^{1/4}}M\kkn \right] 
+ (k+1-tN^{1/2}) S^{k,N}  \nonumber \\
&= S^{k,N} +(t-t_k) \bl^{k,N} + N^{1/4} (t-t_k) M\kkn.\nonumber
\end{align}
Iterating the above we obtain
\begin{align}
 S^{(N)}(t)& =S^{0,N}  + (t-t_k) \bl^{k,N}+\frac{1}{\sqrt{N}} \sum_{j=0}^{k-1}\bl^{j,N}+ {w^N(t)},\nonumber
\end{align}
where 
\be\label{wN}
 w^N(t):=\frac{1}{N^{1/4}}\sum_{j=0}^{k-1}M^{j,N}+N^{1/4}(t-t_k) M\kkn \quad t_k\leq t <t_{k+1}.
\ee
The expression for $S^{(N)}(t)$ can then be rewritten as 
\begin{align}
 S^{(N)}(t) = S^{0,N}+\int_0^t \bl(S^{(N)}(v)) dv+ \hat{w}^N(t),\label{continterforcontmap}
\end{align}
having set 
\be\label{whatN}
\hat{w}^N(t):=e^N(t)+w^N(t),
\ee 
with
\be\label{defennnnnnnn}
e^N(t):=(t-t_k) \bl^{k,N}+\frac{1}{\sqrt{N}} \sum_{j=0}^{k-1}\bl^{j,N}-\int_0^t \bl(S^{(N)}(v)) dv.
\ee
Equation \eqref{continterforcontmap} shows that
 $$
 S^{(N)}=\mathcal{J}_2(S^{0,N},\hat{w}^N),
 $$  
where $\mathcal{J}_2$ is the It\^o map defined in the statement of Theorem \ref{contofupsilon}. By the continuity of the map $\mathcal{J}_2$, if we show that  $\hat{w}^N$ converges in probability  in $C([0,T]; \R)$ to zero,  then $S\bn(t)$ converges in probability to the solution of the ODE \eqref{ODE}.  We prove   convergence of $\hat{w}^N$ to zero in  Section \nolinebreak \ref{sec8}. 
 In view of \eqref{whatN}, we show the convergence in probability of $\hat{w}^N$ to zero  by proving  that both $e^N$ (Lemma \ref{syntesys}) and $w^N$ (Lemma \ref{lem:noiseskn}) converge in $L_2(\Omega; C([0,T]; \R))$ to zero. Because $\{S^{0,N}\}_{N\in\N}$ is a deterministic sequence  that converges to $S^0$, we then have that $(S^{0,N},\hat{w}^N)$ converges in probability to $(S^0,0)$.}

\subsection{Continuous Mapping Argument for \eqref{SPDe1}}\label{subs:cma2}
We now consider  the chain $\{x\kkn\}_{k\in\N}\subseteq \hs$, defined in \eqref{chainxcomponents}.  We act analogously to what we have done for the chain $\{S^{k,N}\}_{k\in\N}$. So we start by recalling the definition of the continuous interpolant $x\bn$, equation \eqref{interpolant} and the notation introduced at the beginning of Section \ref{heuranchainx}. 
{An argument analogous to the one used to derive  \eqref{continterforcontmap} shows that for any $t\in[t_k,t_{k+1})$ 
\begin{align}
x^{(N)}(t)&=x^{0,N}+ (t-t_k) \Theta^{k,N}+ \frac{1}{\sqrt{N}}\sum_{j=0}^k\Theta^{j,N}+{\eta^N(t)}\nonumber\\
& = x^{0,N}+\int_0^t \Theta(x^{(N)}(v),S(v)) dv+ \hat{\eta}^N(t),\label{contmapetahatx}
\end{align}
where 
\begin{align}
\hat{\eta}^N(t)&:=d^N(t)+\upsilon^N(t)+\eta^N(t) \label{etahatN}, \\
{\eta}^N(t)&:={N^{1/4}(t-t_k)L\kkn+\frac{1}{N^{1/4}} \sum_{j=1}^{k-1}L^{j,N}},\label{etaN}
\end{align}
and
\begin{align}
d^N(t)&:= (t-t_k) \Theta^{k,N}+ \frac{1}{\sqrt{N}}\sum_{j=0}^{k-1}\Theta^{j,N} - \int_0^t \Theta(x^{(N)}(v),S^{(N)}(v))dv, \label{defd}\\
\upsilon^N(t)&:=\int_0^t \left[\Theta(x^{(N)}(v),S^{(N)}(v))- \Theta(x^{(N)}(v),S(v))\right] dv. \label{defupsilon}
\end{align}
Equation \eqref{contmapetahatx} implies that 
  \be\label{flu}
  x^{(N)}=\mathcal{J}_1(x^{0,N},\hat{\eta}^N),
  \ee
 where $\mathcal{J}_1$ is It\^o map defined in the statement of Theorem \ref{contofupsilon}. In Section \ref{sec9} we prove that $\hat{\eta}^N$ converges  weakly in $C([0,T];\hs)$ to the process  $\eta$, where the process $\eta$ is the diffusion part of equation \eqref{informalSPDe}, i.e.
\be\label{eq:eta}
\eta(t):=\int_0^t \sqrt{2h\el (S(v))}  dW_v,
\ee
 with $W_v$  a $\hs$-valued $\C_s$-Brownian motion.
 Looking at \eqref{etahatN}, we prove the weak convergence of $\hat{\eta}^N$ to $\eta$ by the following steps:
\begin{enumerate}
\item  We prove that $d^N$ converges in $L_2(\Omega ; C([0,T]; \hs))$ to zero (Lemma \ref{bho});
\item using the convergence in probability (in $C([0,T]; \R)$) of $S^{(N)}$ to $S$, we show convergence in probability (in $C([0,T]; \hs)$) of $\upsilon^N$ to zero (Lemma \ref{lem:driftxkn3});
\item  we show  that $\eta^N$ converges in weakly in $C([0,T]; \hs)$ to  the process $\eta$, defined in \eqref{eq:eta} (Lemma \ref{lem:noise1}).
\end{enumerate}
Because $\{x^{0,N}\}_{N\in\N}$ is a deterministic sequence  that converges to $x^0$, the above three steps (and Slutsky's Theorem) imply that $(x^{0,N},\hat{\eta}^N)$ converges weakly to $(x^0,\eta)$. Now observe that $x(t)=\mathcal{J}_1(x^0, \eta(t))$, where $x(t)$ is the solution of the SDE \eqref{xsolofSPDE}. The continuity of  the map $\mathcal{J}_1$ (Theorem \ref{contofupsilon}), \eqref{flu} and the Continuous Mapping Theorem then imply that the sequence $\{x^{(N)}\}_{N\in\N}$ converges weakly to the solution of the SDE \eqref{xsolofSPDE} (equivalently, to the solution of the SDE \eqref{informalSPDe}),  thus establishing Theorem \ref{thm:mainthm1}.}
\section{Preliminary Estimates and Analysis of the Acceptance Probability}\label{sec7}
This section  gathers several technical results. In Lemma \ref{lem:bla} we study the size of the jumps of the chain. Lemma \ref{lem:moments} contains uniform bounds on the moments of the chains $\{x^{k,N}\}_{k\in\N}$ and $\{S^{k,N}\}_{k\in\N}$, much needed  in  Section \ref{sec8} and  Section \ref{sec9}.   In Section \ref{subsaccprop} we detail  the analysis of the  acceptance probability. This allows us to quantify the correlations between $\gamma^{k,N}$ and the noise $\xi^{k,N}$,  Section \ref{corrsubec}. Throughout the paper, when referring to the function $Q^N$ defined in \eqref{QN2}, we use interchangeably the notation $Q^N(x^{k,N}, y^{k,N})$ and $Q^N(x^{k,N}, \xi^{k,N})$ (as we have already remarked, given $x^{k,N}$, the proposal $y^{k,N}$ is only a function of $\xi^{k,N}$. )
\begin{lemma}\label{lem:bla} Let $q\geq1/2$ be a real number. Under Assumption \ref{ass:1} the following holds:
\be\label{eqlem:bla1}
\EE_k{\norms{y^{k,N}-x^{k,N}}^{2q}}\lesssim  \frac{1}{N^{q/2}}(1+\norms{x^{k,N}}^{2q})
\ee
 and 
\be\label{eqlem:bla2}
 \EE_k{\normcn{y^{k,N}-x^{k,N}}^{2q}}\lesssim (S^{k,N})^q+N^{q/2}.
\ee
Therefore, 
\be\label{eqlem:bla3}
\EE_k{\norms{x^{k+1,N}-x^{k,N}}^{2q}}\lesssim  \frac{1}{N^{q/2}}(1+\norms{x^{k,N}}^{2q}),
\ee
and
\be\label{eqlem:bla4}
 \EE_k{\normcn{x^{k+1,N}-x^{k,N}}^{2q}}\lesssim (S^{k,N})^q+N^{q/2}.
\ee
\end{lemma}
\begin{proof} By definition of the proposal $y^{k,N}$, equation \eqref{eqn:proposal}, 
\begin{align*}
\norms{y^{k,N}-x^{k,N}}^{2q} &= \norms{\delta(x^{k,N}+\C_N\nabla\Psi^N(x^{k,N}))+\sqrt{2\delta}\C_N^{1/2}\xi^{k,N}}^{2q}\\
&\lesssim\frac{1}{N^q}\left(\norms{x^{k,N}}^{2q}+\norms{\C_N\nabla\Psi^N(x^{k,N})}^{2q}\right)+\frac{1}{N^{q/2}}
\norms{\C_N^{1/2}\xi^{k,N}}^{2q}.
\end{align*}
 Thus, using  \eqref{eq:lipz} and \eqref{c1/2xi}, we have 
\begin{align*}
\EE_k\norms{y^{k,N}-x^{k,N}}^{2q}&\lesssim  \frac{1}{N^q}\left(1+\norms{x^{k,N}}^{2q}\right)+\frac{1}{N^{q/2}}\\
&\lesssim \frac{1}{N^{q/2}}\left(1+\norms{x^{k,N}}^{2q}\right),
\end{align*}
which proves \eqref{eqlem:bla1}. Equation \eqref{eqlem:bla2} follows similarly:
\begin{align*}
\EE_k{\normcn{y^{k,N}-x^{k,N}}^{2q}}&\lesssim \frac{1}{N^q}\left(\normcn{x^{k,N}}^{2q}+
\normcn{\C_N\nabla\Psi^N(x^{k,N})}^{2q}\right)\\
&+\frac{1}{N^{q/2}}\Ebc{\normcn{\C_N^{1/2}\xi^{k,N}}^{2q}}.
\end{align*}
Since $\normcn{\C_N^{1/2}\xi^{k,N}}^{2}=\sum_{j=1}^N(\xi^{k,N}_j)^2$ has chi-squared law, applying Stirling's formula for the Gamma function $\Gamma:\R\to\R$ we obtain
\be\label{Stirl}
\EE_k\normcn{\C_N^{1/2}\xi^{k,N}}^{2q}\lesssim\frac{\Gamma(q+N/2)}{\Gamma(N/2)}\lesssim N^q.
\ee
Hence, using  \eqref{eq:B4},  the desired bound follows. Finally,  recalling the  definition of the chain, equation \eqref{chain-gamma},  the bounds  \eqref{eqlem:bla3} and \eqref{eqlem:bla4} are clearly a consequence of \eqref{eqlem:bla1} and \eqref{eqlem:bla2}, respectively, since either $x^{k+1,N}=y^{k,N}$ (if the proposed move is accepted) or $x^{k+1,N}=x^{k,N}$ (if the move is rejected).
\end{proof}

\begin{lemma}\label{lem:moments} If Assumption \ref{ass:1} holds, then, for every $q\geq 1$, we have 
\begin{align}
& \exo(S^{k,N})^q\lesssim 1 \label{eq:bounded2}\\
&\exo{\norms{x^{k,N}}^q}\lesssim 1 \label{lem:boundedness2}, 
\end{align}
uniformly over $N \in \N$ and $k \in \{0, 1 \dd [T\sqrt{N}]\}$. 
\end{lemma}
\begin{proof} The proof of this lemma can be found in Appendix \ref{momproofs}.
\end{proof}
\subsection{Acceptance Probability}\label{subsaccprop}
The main result of this section is Proposition \ref{propac}, which we obtain as a consequence of Lemma \ref{lem:bound2} (below) and Lemma \ref{lem:moments}.  Proposition \ref{propac} formalizes the heuristic approximation \eqref{star}. 
\begin{lemma}[Acceptance probability] \label{lem:bound2} Let Assumption \ref{ass:1} hold and recall the definitions \eqref{accprobQQ} and \eqref{def:al}. Then the following holds:
$$\Ebc{\lv\alpha^N(x^{k,N},\xi^{k,N})-\alpha_\ell(S^{k,N})\rv^{2}}\lesssim \frac{1+(S^{k,N})^2+\norms{x^{k,N}}^2}{\sqrt{N}}.$$
\end{lemma}
Before proving Lemma \ref{lem:bound2} , we state Proposition \ref{propac}.
\begin{prop}\label{propac}
If Assumption \ref{ass:1} holds then 
$$
\lim_{N\ra \infty}\exo{\lv\alpha^N(x^{k,N},y^{k,N})-\alpha_\ell(S^{k,N})\rv^{2}}=0.
$$
\end{prop}
\begin{proof}
This is a corollary of  Lemma \ref{lem:bound2} 
and Lemma \ref{lem:moments}. 
\end{proof}
\begin{proof}[Proof of Lemma \ref{lem:bound2}]The function $z\mapsto 1\wedge e^z$ on $\R$ is globally Lipschitz with Lipschitz constant $1$. Therefore, by \eqref{def:al} and \eqref{accprobQQ},
$$\Ebc{\lv\alpha^N(x^{k,N},y^{k,N})-\alpha_\ell(S^{k,N})\rv^{2}}\leq \Ebc{\lv Q^N(x^{k,N},y^{k,N})-\frac{\ell^2(S^{k,N}-1)}{2}\rv^{2}}.$$
The result is now a consequence of  \eqref{formap} below.
\end{proof}
To analyse the acceptance probability it is convenient to decompose $Q^N$ as follows:
\be\label{QNdecomp}
Q^N(x^N,y^N)=I_1^N(x^N,y^N)+I_2^N(x^N,y^N)+I_3^N(x^N,y^N)
\ee
where
\begin{align}
I_1^N(x^N,y^N)&:=-\frac{1}{2}\left[\norm{y^N}_{\C_N}^2-\norm{x^N}_{\C_N}^2\right]-\frac{1}{4\delta}
\left[\normcn{x^N-(1-\delta)y^N}^2-\normcn{y^N-(1-\delta)x^N}^2\right]\nonumber\\
&=-\frac{\delta}{4}(\normcn{y^N}^2-\normcn{x^N}^2),\label{A1}\\
I_2^N(x^N,y^N)&:=-\frac{1}{2}\left[\iprodcn{x^N-(1-\delta)y^N}{\C_N\nabla\Psi^N(y^N)}-
\iprodcn{y^N-(1-\delta)x^N}{\C_N\nabla\Psi^N(x^N)}\right]\nonumber\\
&-(\Psi^N(y^N)-\Psi^N(x^N)),\label{A2}\\
I_3^N(x^N,y^N)&:=-\frac{\delta}{4}\left[\normcn{\C_N\nabla\Psi^N(y^N)}^2-\normcn{\C_N\nabla\Psi^N(x^N)}^2\right]. \label{A3}
\end{align}
\begin{lemma}\label{lem:lem:I1}Let Assumption \ref{ass:1} hold. With the notation introduced above, we have:
\begin{align}
&   \Ebc{\mmag{I_1^N(x^{k,N},y^{k,N})-\frac{\ell^2(S^{k,N}-1)}{2}}^2}
\lesssim \frac{\norms{x^{k,N}}^2}{N^2}+\frac{(S^{k,N})^2}{\sqrt{N}}+\frac{1}{N} \label{lem:I1}\\
& \Ebc{\mmag{I_2^N(x^{k,N},y^{k,N})}^2}\lesssim \frac{1+\norms{x^{k,N}}^2}{\sqrt{N}} \label{lem:I2}\\
& \Ebc{\mmag{I_3^N(x^{k,N},y^{k,N})}^2}\lesssim \frac{1}{N}.
\label{lem:I3}
\end{align}
Therefore, 
\be\label{formap}
\Ebc{\lv Q^N(x^{k,N},y^{k,N})-\frac{\ell^2(S^{k,N}-1)}{2}\rv^{2}}\les 
\frac{1+(S^{k,N})^2+\norms{x^{k,N}}^2}{\sqrt{N}}. 
\ee
\end{lemma}

\begin{proof}
We consecutively prove the three bounds in the statement. 
\begin{description} 
\item[$\bullet$ Proof of \eqref{lem:I1}. ] Using \eqref{eqn:proposal},  we rewrite $I_1^N$ as
$$I_1^N(x^{k,N},y^{k,N})=-\frac{\delta}{4}\left(\normcn{(1-\delta)x^{k,N}-\delta\C_N\nabla\Psi^N(x^{k,N})+\sqrt{2\delta}
\C_N^{1/2}\xi^{k,N}}^2-\normcn{x^{k,N}}^2\right).$$
Expanding the above we obtain:
\begin{align}
I_1^N(x^{k,N},y^{k,N})-\frac{\ell^2(S^{k,N}-1)}{2} &= 
-\left(\frac{\delta^2}{2}\normcn{\C_N^{1/2}\xi^{k,N}}^2-\frac{\ell^2}{2}\right) \nonumber\\
&+(r_{\Psi}^N - r^N)+r_{\xi}^N+r_x^N, \label{I1expanded}
\end{align}
where the difference $(r_{\Psi}^N - r^N)$ is defined in \eqref{diff} and we set
\begin{align}
r^N_{\xi}&:= -\frac{(\delta^{3/2}-\delta^{5/2})}{\sqrt{2}}\iprodcn{x^{k,N}}{\C_N^{1/2}\xi^{k,N}},\label{rxi}\\
r^N_{x}&:= -\frac{\delta^3}{4}\normcn{x^{k,N}}^2\label{rx}.
\end{align} 
For the reader's convenience we rearrange \eqref{diff} below:
\begin{align}
r_{\Psi}^N - r^N&= \frac{\delta^2-\delta^3}{2}\iprodcn{x^{k,N}}{\C_N\nabla\Psi^N(x^{k,N})} \nonumber\\
&-\frac{\delta^3}{4}\normcn{\C_N\nabla\Psi^N(x^{k,N})}^2+\frac{\delta^{5/2}}{\sqrt{2}}\iprodcn{\C_N\nabla\Psi^N(x^{k,N})}
{\C_N^{1/2}\xi^{k,N}}. \label{rp-r}
\end{align}
We come to bound all of the above terms,  starting from \eqref{rp-r}.  To this end, let us observe the following:
\begin{align}\label{eq:29a}
\mmag{\iprodcn{x^{k,N}}{\C_N\nabla\Psi^N(x^{k,N})}}^2& =\lv \sum_{i=1}^N
x^{k,N}_i [\nabla \Psi^N(x^{k,N})]_i\rv^2 \\
& \stackrel{\eqref{eq:B2}}
{\leq} \nors{x^{k,N}}^2 \|\nabla \Psi^N(x^{k,N})\|_{-s}^2 \stackrel{\eqref{eq:lin}}{\les} \nors{x^{k,N}}^2.
\end{align}
Moreover, 
$$
\EE_k \norcN{\C_N^{1/2} \xi^{k,N}}^2 = \EE_k \sum_{j=1}^N \lv \xi_j\rv^2 = N,
$$
hence
$$
\lv \iprodcn{\C_N\nabla\Psi^N(x^{k,N})}
{\C_N^{1/2}\xi^{k,N}}\rv^2 \leq \norcN{\C_N\nabla\Psi^N(x^{k,N})}^2\norcN{\C_N^{1/2} \xi^{k,N}}^2
\stackrel{\eqref{eq:B4}}{\les}N. 
$$
From \eqref{rp-r}, \eqref{eq:29a}, \eqref{eq:B4} and the above, 
\be\label{estrps-r}
\EE_k \lv r_{\Psi}^N-r^N\rv^2 \les \frac{\nors{x\kkn}^2}{N^2}+\frac{1}{N^{3/2}}. \ee
By \eqref{rxi}, 
\begin{align}
\EE_k \lv r^N_{\xi}\rv^2 &\les \frac{1}{N^{3/2}} \EE_k\lv\iprodcn{x^{k,N}}{\C_N^{1/2}\xi^{k,N}}\rv^2\nonumber\\
& = \frac{1}{N^{3/2}}\EE_k \left(\sum_{i=1}^N  \frac{x_i^{k,N} \xi_i\kkn}{\lambda_i} \right)^2 = \frac{1}{\sqrt{N}}S\kkn,  
\end{align}
where in the last equality we have used the fact that $\{\xi_i^{k,N}:i=1,\dots,N\}$ are independent, zero mean, unit variance normal random variables (independent of $x^{k,N}$) and \eqref{defskn}.
As for $r^N_{x}$, 
$$
\EE_k \lv r_x^N\rv^2 \les \frac{1}{N^3}\normcn{x^{k,N}}^4\stackrel{\eqref{defskn}}{=}\frac{(S\kkn)^2}{N}.
$$ 
Lastly, 
$$\tilde{r}^N:=\frac{\delta^2}{2}\normcn{\C_N^{1/2}\xi^{k,N}}^2-\frac{\ell^2}{2}=\frac{\ell^2}{2}\left( \frac{1}{N}\sum_{j=1}^N\xi^2_j-1\right).$$
Since $\sum_{j=1}^N\xi^2_j$ has chi-squared law, $\EE_k\lv \tilde{r}^N\rv^2\lesssim Var\left( N^{-1}\sum_{j=1}^N\xi^2_j\right)\lesssim N^{-1}$, by \eqref{Stirl}. Combining all of the above,  we obtain the desired bound.
\item[$\bullet$ Proof of \eqref{lem:I2}] From \eqref{A2}, 
\begin{align*}
I_2^N(x^{k,N},y^{k,N})=&-\left[\Psi^N(y^{k,N})-\Psi^N(x^{k,N})-\iprod{y^{k,N}-x^{k,N}}{\nabla\Psi^N(x^{k,N})}\right]\\ 
&+\frac{1}{2}\iprod{y^{k,N}-x^{k,N}}{\nabla\Psi^N(y^{k,N})-\nabla\Psi^N(x^{k,N})} \\
&+\frac{\delta}{2}\left(\iprod{x^{k,N}}{\nabla\Psi^N(x^{k,N})}-\iprod{y^{k,N}}{\nabla\Psi^N(y^{k,N})}\right)
=:\sum_{j=1}^3d_j,
\end{align*}
where $d_j$ is the addend on line $j$ of the above array. 
Using \eqref{eq:taylor}, \eqref{eq:lin}, \eqref{eq:B2} and Lemma \ref{lem:bla}, we have 
$$
\Ebc\lv d_1\rv^{2}\lesssim\Ebc{\norm{y^{k,N}-x^{k,N}}_s^{2}}\lesssim\frac{1+\norms{x^{k,N}}^{2}}{\sqrt{N}}.$$
By the first inequality in  \eqref{eq:lin}, 
$$\normms{\nabla\Psi^N(y^{k,N})-\nabla\Psi^N(x^{k,N})}\lesssim 1.$$
Consequently,  again by \eqref{eq:B2} and Lemma \ref{lem:bla}, 
$$
\Ebc\lv d_2\rv^{2}\lesssim\Ebc{\norm{y^{k,N}-x^{k,N}}_s^{2}}\lesssim\frac{1+\norms{x^{k,N}}^{2}}{\sqrt{N}}.$$
 Next, applying \eqref{eq:B2} and \eqref{eq:lin} gives%
\begin{align*}
\lv{d_3}\rv&\leq\frac{\norms{x^{k,N}}\normms{\nabla\Psi^N(x^{k,N})}+\norms{y^{k,N}}\normms{\nabla\Psi^N(y^{k,N})}}{\sqrt{N}}\\
&\lesssim\frac{\norms{x^{k,N}}+\norms{y^{k,N}}}{\sqrt{N}}
\les \frac{\norms{x^{k,N}}+\norms{y^{k,N}-x^{k,N}}}{\sqrt{N}}.
\end{align*}
Thus,  applying Lemma \ref{lem:bla} then gives the desired bound.
\item[$\bullet$ Proof of \eqref{lem:I3}] This follows directly from \eqref{eq:lipz}.
\end{description}
\end{proof}

\subsection{Correlations Between Acceptance Probability and Noise $\xi^{k,N}$}\label{corrsubec}
Recall the definition of $\gamma^{k,N}$, equation \eqref{defgammaaccept},  and let
\be\label{eps}
\varepsilon^{k,N}:= \gamma^{k,N}\C_N^{1/2}\xi^{k,N}.
\ee
The study of the properties of $\varepsilon^{k.N}$ is the object of the next two lemmata, which have a central role in the analysis:  Lemma \ref{lem:epsilon} (and Lemma \ref{lem:moments}) establishes  the decay of correlations between the acceptance probability and the noise $\xi^{k,N}$. Lemma \ref{lem:AI} formalizes the heuristic arguments presented in Section \ref{subs:addx}.

\begin{lemma}\label{lem:epsilon} 
If Assumption \ref{ass:1} holds, then
\be\label{normeps}
\norms{{\Ebc\varepsilon\kkn}}^2\lesssim\frac{1+\norms{x^{k,N}}^2}{\sqrt{N}}.
\ee
Therefore, 
\be\label{circle}
\iprods{{\Ebc\varepsilon\kkn}}{x\kkn}{=\EE_k\iprods{ \gamma^{k,N}\C_N^{1/2}\xi^{k,N}}{x^{k,N}}}\lesssim \frac{1}{N^{1/4}}(1+\norms{x^{k,N}}^2).
\ee
\end{lemma}
{\begin{lemma}\label{lem:AI} Let Assumption \ref{ass:1} hold. Then, 
with the notation introduced so far, 
$$ \lim_{N\ra \infty}\exo\mmag{\EE_k \norms{\varepsilon\kkn}^2-\tr_{\h^s}(\C_s)\al(S^{k,N})}=0.$$
\end{lemma}}
The proofs of the above lemmata can be found in  Appendix \ref{corrproofs}. Notice that if $\xi^{k,N}$ and $\gamma^{k,N}$ (equivalently $\xi^{k,N}$ and $Q^{N}$) were uncorrelated, the statements of Lemma \ref{lem:epsilon} and Lemma \ref{lem:AI} would be trivially true.

\section{Proof of Theorem \ref{thm:weak conv of Skn}}\label{sec8}

{As explained in Section \ref{sebs:contmaparg2},  due to the continuity of the map $\mathcal{J}_2$ (defined in Theorem \ref{contofupsilon}), in order to prove Theorem \ref{thm:weak conv of Skn} all we need to show is convergence in probability of $\hat{w}^N(t)$ to zero. Looking at the definition of $\hat{w}^N(t)$, equation \eqref{whatN}, the convergence in probability (in $C([0,T];\R)$) of $\hat{w}^N(t)$ to zero is consequence of Lemma \ref{syntesys}  and Lemma \ref{lem:noiseskn} below. We prove Lemma \ref{syntesys} in Section \ref{driftSkN} and Lemma \ref{lem:noiseskn} in Section \ref{sec:noiseSkN}.} 

\begin{lemma}\label{syntesys}
Let Assumption \ref{ass:1} hold and recall the definition \eqref{defennnnnnnn} of the process $e^N(t)$; then
$$
\lim_{N\to\infty}\exo\left(\sup_{t\in[0,T]}\mmag{e^N(t)}\right)^2=0.
$$
\end{lemma}

\begin{lemma}\label{lem:noiseskn}
Let Assumption \ref{ass:1} hold and recall the definition \eqref{wN} of the process $w^N(t)$; then
$$
\lim_{N\to\infty}\exo\left(\sup_{t\in[0,T]}\mmag{w^N(t)}\right)^2=0.
$$
\end{lemma}

\subsection{Analysis of the Drift}\label{driftSkN}

In view of what follows, it is convenient to introduce the piecewise constant interpolant of the chain $\{S^{k,N}\}_{k\in\N}$:
\be\label{piecconstinterSkn}
\bar{S}^{(N)}(t):=S^{k,N}, \quad t_k\leq t<t_{k+1}, 
\ee
 where $t_k= k/\sqrt{N}$. 
 \begin{proof}[Proof of Lemma \ref{syntesys}]
 From \eqref{piecconstinterSkn},  for any $t_k\leq t<t_{k+1}$ we have
 \begin{align*}
 \int_0^t\bl(\bar{S}^{(N)}_v)dv &= \int_{t_k}^t\bl(\bar{S}^{(N)}_v)dv
 +\sum_{j=1}^{k-1}\int_{t_{j-1}}^{t_j}\bl(\bar{S}^{(N)}_v)dv\\
& =(t-t_k)\bl(S^{k,N})+ \frac{1}{\sqrt{N}}\sum_{j=1}^{k-1}\bl(S^{j,N}).
\end{align*}
 With this observation, we can then  decompose  $e^N(t)$ as
$$
e^N(t)=e^N_1(t)- e^N_2(t), 
$$
where
\begin{align}
e^N_1(t)&:=(t-t_k) (\bl^{k,N}-\bl(S^{k,N}))+\frac{1}{\sqrt{N}} \sum_{j=0}^{k-1}\left[
\bl^{j,N}-\bl(S^{j,N})\right]
\label{en1}\\
e_2^N(t)& :=\int_0^t \left[\bl( S^{(N)}_v)-\bl(\bar{S}^{(N)}_v) \right]dv. \label{en1}
\end{align}
The result is now a consequence of  Lemma \ref{lem:skn1} and Lemma \ref{lem:skn2} below, which we first state and then consecutively prove.
\end{proof}
\begin{lemma}\label{lem:skn1}If Assumption \ref{ass:1} holds, then
{$$\lim_{N\to\infty}\exo\left(\sup_{t\in[0,T]}
\lv e_1^N(t) \rv\right)^2=0.$$}
\end{lemma}

\begin{lemma}\label{lem:skn2}If Assumption \ref{ass:1} holds, then
$$
\lim_{N\to\infty}\exo\left(\sup_{t\in[0,T]}
\lv e_2^N(t) \rv\right)^2=0.
$$
\end{lemma}
\begin{proof}[Proof of Lemma \ref{lem:skn1}] Denoting $E^{k,N}:=\bl^{k,N}-b_\ell(S^{k,N})$,  by Jensen's inequality we have
{
\begin{align*}
\sup_{t\in[0,T]}\mmag{e_1^N(t)}^2 & =\sup_{t\in[0,T]}\mmag{(t-t_k) E^{k,N}+\frac{1}{\sqrt{N}} \sum_{j=0}^{k-1}E^{k,N}}^2\\
&\les \frac{1}{\sqrt{N}} \sum_{j=0}^{[T\sqrt{N}]-1}\mmag{E^{j,N}}^2.
\end{align*}
 Using Lemma \ref{lem:driftskn} below, we obtain
$$\frac{1}{\sqrt{N}} \sum_{j=0}^{[T\sqrt{N}]-1}\mmag{E^{j,N}}^2\lesssim\frac{1}{\sqrt{N}}\sum_{k=0}^{[T\sqrt{N}]-1}\frac{1+(S^{k,N})^4+\norms{x^{k,N}}^4}{\sqrt{N}}.$$
Taking expectations on both sides and applying Lemma \ref{lem:moments} completes the proof.
}
\end{proof}
\begin{lemma} \label{lem:driftskn} Let Assumption \ref{ass:1} hold. Then, for any $N \in \N$ and  $k\in\{0, 1 \dd [T\sqrt{N}]\}$, 
$$\lv E^{k,N} \rv^2=
\lv\bl^{k,N}-\bl(S^{k,N})\rv^2\lesssim \frac{1+(S^{k,N})^4+\norms{x^{k,N}}^4}{\sqrt{N}}.$$
\end{lemma}
\begin{proof}
Define
$$Y^N_k:=\frac{\normcn{y^{k,N}}^2 -\normcn{x^{k,N}}^2}{\sqrt{N}},\qquad \tilde{Y}^N_k:=2\ell (1-S^{k,N}).$$
Then, from \eqref{starstar},  \eqref{accprobQQ}, \eqref{def:al} and \eqref{defbl},  we obtain
\begin{align*}
\lv\bl^{k,N}-\bl(S^{k,N})\rv^2&=\lv\Ebc{\left(\alpha^N(x^{k,N},y^{k,N})Y^N_k}\right)-  \alpha_\ell(S^{k,N})\tilde{Y}^N_k\rv^2 \\ &\leq \Ebc{\lv \alpha^N(x^{k,N},y^{k,N})Y^N_k-  \alpha_\ell(S^{k,N})\tilde{Y}^N_k \rv^2}\\
&\lesssim \Ebc{\left[\lv \alpha^N(x^{k,N},y^{k,N})\rv^2\lv Y^N_k-\tilde{Y}^N_k \rv^2}\right]\\ &+\Ebc{\left[\lv \tilde{Y}^N_k \rv^2\lv \alpha^N(x^{k,N},y^{k,N})-  \alpha_\ell(S^{k,N})\rv^2}\right].
\end{align*}
Since $|\alpha^N(x^{k,N},y^{k,N})|\leq 1$ and $\tilde{Y}^N_k$ is a function of $x^{k,N}$ only, we can further estimate the above as follows:
\be\label{heart}
\lv\bl^{k,N}-\bl(S^{k,N})\rv^2\lesssim \Ebc{\lv Y^N_k-\tilde{Y}^N_k \rv^2}+\lv \tilde{Y}^N_k\rv^2\Ebc{
\lv\alpha^N(x^{k,N},y^{k,N})-  \al(S^{k,N})\rv^2}.
\ee
From the definition of $I_1^N$, equation \eqref{A1}, we have
\be\label{D}
Y\kkn=-\frac{4}{\ell}I_1^N(x^{k,N},y^{k,N}).
\ee
Therefore,  
$$
Y^N_k-\tilde{Y}^N_k= -\frac{4}{\ell} \left[I_1^N - \frac{\ell^2}{2}(S^{k,N}-1) \right],
$$
which implies
$$\Ebc{(Y^N_k-\tilde{Y}^N_k)^2}\lesssim\Ebc{\left(I_1^N(x^{k,N},y^{k,N})-\ell^2 (S^{k,N}-1)/2\right)^2}\stackrel{\eqref{lem:I1}}{\lesssim}\frac{\norms{x^{k,N}}^2}{N^2}+\frac{(S^{k,N})^2}{\sqrt{N}}+\frac{1}{N}.$$
As for  the second addend in \eqref{heart}, Lemma \ref{lem:bound2} gives
\begin{align*}\lv \tilde{Y}^N_k\rv^2\Ebc{\lv\alpha^N(x^{k,N},y^{k,N})-  \alpha_\ell(S^{k,N})\rv^2}& \lesssim (1+(S^{k,N})^2)\left(\frac{1+(S^{k,N})^2+\norms{x^{k,N}}^2}{\sqrt{N}}\right) \\ &\lesssim \frac{1+(S^{k,N})^4+\norms{x^{k,N}}^4}{\sqrt{N}}.\end{align*}
Combining the above two bounds and \eqref{heart} gives  the desired result.
\end{proof}
\begin{proof}[Proof of Lemma \ref{lem:skn2}] {By Jensen's inequality, 
\begin{align*}
\left(\sup_{t\in[0,T]}\mmag{\int_0^t\bl( S^{(N)}_v)-\bl(\bar{S}^{(N)}_v)dv}\right)^2\lesssim \int_0^T\mmag{\bl( S^{(N)}_v)-\bl(\bar{S}^{(N)}_v)}^2dv.
\end{align*}
}
Since $\bl$ is globally Lipschitz, 
\begin{align*}\int_0^T\mmag{\bl (\bar{S}^N(v)) -\bl(S^N(v)) }^2dv & \les\int_0^T\mmag{\bar{S}^N(v) -S^N(v) }^2dv\\
&=\sum_{k=0}^{[T\sqrt{N}]-1}\int_{t_k}^{t_{k+1}}\!\!\!\!\mmag{\bar{S}^N(v) -S^N(v)}^2dv
+\int_{[T\sqrt{N}]}^{T}\mmag{\bar{S}^N(v) -S^N(v)}^2dv\\
&\lesssim\frac{1}{\sqrt{N}}\sum_{k=0}^{[T\sqrt{N}]-1}(S^{k+1,N}-S^{k,N})^2.\end{align*}
From \eqref{C} and \eqref{defskn}, 
\begin{align*}
\lv S^{k+1,N}-S^{k,N}\rv & \les \frac{1}{N} \left(
\| y^{k,N}\|_{\C^N}^2- \| x^{k,N}\|_{\C^N}^2 \right)\\
&\stackrel{\eqref{D}}{\les} \frac{1}{\sqrt{N}} I_1^N(x^{k,N},y^{k,N})\\
& =
\frac{1}{\sqrt{N}} \left( I_1^N(x^{k,N},y^{k,N})  - \frac{\ell^2 (S^{k,N}-1)}{2} \right) +
\frac{1}{\sqrt{N}}\frac{\ell^2 (S^{k,N}-1)}{2} \,.
\end{align*}
Combining the above  with \eqref{lem:I1} we obtain
\be\label{eskp1-sk}
\EE_k  {(S^{k+1,N}-S^{k,N})^2} \les \frac{1+(S^{k,N})^2+\norms{x^{k,N}}^2}{N}\,.
\ee
Taking expectations and applying Lemma \ref{lem:moments} concludes the proof.
\end{proof}
\subsection{Analysis of Noise}\label{sec:noiseSkN}
\begin{proof}[Proof of Lemma \ref{lem:noiseskn}] After a calculation analogous to the one at the beginning of the proof of Lemma \ref{lem:skn1}, all we need to prove is the following limit:
$$\frac{1}{\sqrt{N}}\sum_{k=0}^{[T\sqrt{N}]}\exo\lv{M\kkn}\rv^2\to 0\quad\text{as}\quad N\to\infty.$$
By the definition of $M\kkn$, equation \eqref{adsk11111},  we have  
\begin{align*}
\frac{\exo\lv M\kkn\rv^2}{\sqrt{N}}&=\exo\left[S^{k+1,N}-S^{k,N} -\Ebc\left({S^{k+1,N}-S^{k,N}}\right)\right]^2\\
&\lesssim \exo\lv S^{k+1,N}-S^{k,N} \rv^2{\lesssim}\frac{1}{{N}}, 
\end{align*}
where the last inequality is a consequence of \eqref{eskp1-sk} and 
Lemma \ref{lem:moments}.  
This concludes the proof.
\end{proof}

\section{Proof of Theorem \ref{thm:mainthm1}}\label{sec9}

The idea behind the proof is the same as in the previous Section \ref{sec8}. {First we introduce the piecewise constant interpolant of the chain $\{x^{k,N}\}_{k\in\N}$ 
\be\label{piecewiseconstinterx}
\bar{x}\bn(t)=x^{k,N} \quad \mbox{for }\,\, t_k\leq t < t_{k+1}.
\ee}
Due to the continuity of the map $\mathcal{J}_1$ (Theorem \ref{contofupsilon}), all we need to prove is the weak convergence of $\hat{\eta}^N(t)$ to zero (see Section \ref{subs:cma2}). Looking at the definition of $\hat{\eta}^N(t)$, equation \eqref{etahatN}, 
this follows from Lemmas \ref{bho}, \ref{lem:driftxkn3} and  \ref{lem:noise1} below.    We prove Lemma \ref{bho} and  Lemma \ref{lem:driftxkn3} in Section \ref{driftxkN} and Lemma \ref{lem:noise1} in Section \ref{sec:noisexkN}. 

\begin{lemma}\label{bho}
Let Assumption \ref{ass:1} hold and recall the definition \eqref{defd} of the process $d^N(t)$; then
$$
\lim_{N\to\infty}\exo\left(\sup_{t\in[0,T]}\mmag{d^N(t)}\right)^2=0.
$$
\end{lemma}

{\begin{lemma}\label{lem:driftxkn3}If Assumption \ref{ass:1} holds, then $\upsilon^N$ (defined in \eqref{defupsilon}) converges in probability in $C([0,T]; \hs)$ to  zero.
\end{lemma}}
\begin{lemma}\label{lem:noise1}
 Let Assumption \ref{ass:1} hold. Then the interpolated martingale difference array $\mathfrak{\eta}^N(t)$ defined in \eqref{etaN} converges weakly in $C([0,T]; \hs)$ to the stochastic integral  $\eta(t)$, defined in equation \eqref{eq:eta}. 
\end{lemma}
\subsection{Analysis of Drift}\label{driftxkN}
\begin{proof}[Proof of Lemma \ref{bho}]  For all $t\in[t_k,t_{k+1})$, we can write
$$(t-t_k) \Theta(x^{k,N},S^{k,N})+\frac{1}{\sqrt{N}}\sum_{j=0}^{k-1}\Theta(x^{j,N},S^{j,N}) =\int_0^t\Theta(\bar{x}^{(N)}(v),\bar{S}^{(N)}(v))dv.
$$
Therefore, we can decompose $d^N(t)$ as 
$$
d^N(t)=d^N_1(t)+d^N_2(t),  
$$
where 
$$
d_1^N(t):=(t-t_k)\left[\Theta^{k,N}-\Theta(x^{k,N},S^{k,N})\right]
+\frac{1}{\sqrt{N}}\sum_{j=0}^{k-1}\left[\Theta^{j,N}-\Theta(x^{j,N},S^{j,N})\right]
$$
and 
$$
d_2^N(t):=\int_0^t \left[\Theta(\bar{x}^N(v), \bar{S}^N (v)) - \Theta({x}^{(N)}(v), {S}^{(N)}(v))\right] dv.
$$
The statement is now a consequence of Lemma \ref{lem:driftxkn1} and Lemma \ref{lem:driftxkn2}. 
\end{proof}
\begin{lemma}\label{lem:driftxkn1}If Assumption \ref{ass:1} holds, then
$$
\lim_{N\to\infty}\exo\left(\sup_{t\in[0,T]}\norms{d_1^N(t)} \right)^2= 0.
$$
\end{lemma}
\begin{lemma}\label{lem:driftxkn2}If Assumption \ref{ass:1} holds, then
$$\lim_{N\to\infty}\exo\left(\sup_{t\in[0,T]}\norms{
d_2^N(t)} \right)^2= 0.
$$
\end{lemma}
{Before proving  Lemma \ref{lem:driftxkn1},  we state and prove the following Lemma \ref{lem:driftxkn}.  We then consecutively prove Lemma \ref{lem:driftxkn1}, Lemma \ref{lem:driftxkn2} and Lemma \ref{lem:driftxkn3}.   Recall the definitions of $\Theta$ and $\Theta^{k,N}$, equations  \eqref{Theta} and \eqref{approximatedriftd}, respectively. }
\begin{lemma}\label{lem:driftxkn} Let Assumption \ref{ass:1} hold and set \be\label{E}
p\kkn:=\Theta^{k,N}-\Theta(x^{k,N},S^{k,N}).
\ee 
Then
\begin{align*}
\exo\norms{p\kkn}^2\lesssim  &\sum_{j=N+1}^\infty(\lambda_jj^s)^4+\frac{1}{\sqrt{N}}.
\end{align*}
\end{lemma}
\begin{proof} Recalling  \eqref{bit2apprdrift} and   \eqref{eps}, we have  
\begin{align}
\norms{p\kkn}^2
&\lesssim  \sqrt{N}\norms{\EE_k\varepsilon^N_k(x^{k,N})}^2 \label{rhsof}\\
& + \norms{\alpha_\ell(S^{k,N})F(x^{k,N})-
\left[\Ebc{\alpha^N(x^{k,N},y^{k,N})\right](x^{k,N}+\C_N\nabla\Psi^N(x^{k,N}))}}^2,  \label{estest}
\end{align}
where the function $F$ that appears in the above has been defined in Lemma \ref{lem:lipschitz+taylor}. 
The term on the RHS of \eqref{rhsof} has been studied in Lemma \ref{lem:epsilon}.   To estimate the addend in \eqref{estest} we  use \eqref{eq:lipz}, the boundedness of $\al$ and  Lemma \ref{lem:bound2}. A straightforward calculation then gives
\begin{align*}
\eqref{estest}&\les  \left[ \alpha_\ell(S^{k,N}) - \Ebc{\alpha^N(x^{k,N},y^{k,N})}\right]^2\nors{
(x^{k,N}+\C_N\nabla\Psi^N(x^{k,N}))}^2\\
&+ \nors{\al(S^{k,N}) \left[ F(x^{k,N}) -  (x^{k,N}+\C_N\nabla\Psi^N(x^{k,N})) \right]}^2\\
& \lesssim \frac{1+(S^{k,N})^4+\norms{x^{k,N}}^4}{\sqrt{N}} +\norms{\C\nabla\Psi(x^{k,N})-\C_N\nabla\Psi^N(x^{k,N})}^2.
\end{align*}
From the definition of  $\Psi^N$ and $\nabla\Psi^N$, equation \eqref{defpsiNCN} and equation \eqref{gradpsiN}, respectively, 
\begin{align*}
\norms{\C\nabla\Psi(x^{k,N})-\C_N\nabla\Psi^N(x^{k,N})}^2 & = \norms{\C\nabla\Psi(x^{k,N})-\C_ N \PP^N(\nabla\Psi(x^{k,N}))}^2 \\
&=\sum_{j=N+1}^\infty(\lambda_jj^s)^4\Ebb{j^{-2s}(\nabla\Psi(x^{k,N}))_j^2}\lesssim \sum_{j=N+1}^\infty(\lambda_jj^s)^4,
\end{align*}
having used \eqref{eq:lin} in the last inequality.  The statement is now a consequence of Lemma  \ref{lem:moments}. 
\end{proof}
\begin{proof}[Proof of Lemma \ref{lem:driftxkn1}] {Following the analogous steps to those taken in the proof of Lemma \ref{lem:skn1}, the proof} is a direct consequence of Lemma  \ref{lem:driftxkn}, after observing that the summation 
$\sum_{j=N+1}^\infty(\lambda_jj^s)^4$  is the tail of a convergent series hence it tends to zero as $N \ra \infty$. 
%
\end{proof}
\begin{proof}[Proof of Lemma \ref{lem:driftxkn2}]By the definition of $\Theta$, equation \eqref{Theta},  we have 
$$\norms{\Theta(\bar{x}^N(t), \bar{S}^N (t)) - \Theta({x}^N(t), {S}^N (t)) }=\norms{F(\bar{x}^N)h_{\ell}(\bar{S}^N) - F({x}^{(N)})h_{\ell}({S}^{(N)}) }.$$
Applying \eqref{eq:lipz2} and \eqref{eq:lipz} and using the fact $h_\ell$ is globally Lipschitz and bounded,   we get
$$\norms{\Theta(\bar{x}^N(t), \bar{S}^N (t)) - \Theta({x}^N(t), {S}^N (t)) }\lesssim \norms{\bar{x}^N(t)-{x}^{(N)}(t)}+(1+\norms{\bar{x}^N(t)})\mmag{\bar{S}^N (t)-S^{(N)} (t)}.$$
Thus, from the definitions \eqref{interpolantofsk}, \eqref{piecconstinterSkn}, \eqref{interpolant} and \eqref{piecewiseconstinterx},  if  $t_k\leq t<t_{k+1}$, we have 
\begin{align*}
\norms{\Theta(\bar{x}^N(t), \bar{S}^N (t)) - \Theta({x}^N(t), {S}^N (t)) }&\lesssim (t-k\sqrt{N})\norms{x^{k+1,N}-x^{k,N}}\\
&+(t-k\sqrt{N})(1+\norms{x^{k,N}})\mmag{S^{k+1,N}-S^{k,N}}.
\end{align*}
Applying \eqref{eqlem:bla3} and \eqref{eskp1-sk} one then concludes
$$\Ebc{\norms{\Theta(\bar{x}^N(t), \bar{S}^N (t)) - \Theta({x}^N(t), {S}^N (t)) }^2}\lesssim (t-k\sqrt{N})^2\left(\frac{1+\norms{x^{k,N}}^2}{\sqrt{N}}+\frac{\norms{x^{k,N}}^4+(S^{k,N})^4}{N}\right)$$
The remainder of the proof is analogous to the proof of  Lemma \ref{lem:skn2}.
\end{proof}
{\begin{proof}[Lemma \ref{lem:driftxkn3}] For  any arbitrary but fixed  $\varepsilon>0$, we need to argue that 
$$\lim_{N\to\infty}\Pb\left[\sup_{t\in[0,T]}\norms{\upsilon^N(t)}\geq\varepsilon\right]=0.$$
From the definition of $\upsilon^N$ we have 
$$\sup_{t\in[0,T]}\norms{\upsilon^N(t)}\leq \int_0^T \norms{F(x^{(N)}(v))}\mmag{S^{(N)}(v)-S(v)}dv.$$
Using \eqref{e.Flipshitz} and the fact that $\norms{x^{(N)}(t)}\leq \norms{x^{k,N}}+\norms{x^{k+1,N}}$ (which is a simple consequence of \eqref{interpolant}), for any $t\in[t_k,t_{k+1})$
\begin{align*}\sup_{t\in[0,T]}\norms{\upsilon^N(t)}&\leq \left(\sup_{t\in[0,T]}\mmag{S^{(N)}(t)-S(t)}\right)\int_0^T\norms{F(x^{(N)}(v))}dv\\
&\lesssim \underbrace{\left(\sup_{t\in[0,T]}\mmag{S^{(N)}(t)-S(t)}\right)}_{=:a^N}\underbrace{\left(1+\frac{1}{\sqrt{N}}\sum_{j=0}^{[T\sqrt{N}]-1}\norms{x^{j,N}}\right)}_{=:u^N}.
\end{align*}
Using Markov's inequality and Lemma \ref{lem:moments}, given any $\delta>0$, it is straightforward to find constant $M$ such that $\Pb\left[u^N> M\right]\leq \delta$ for every $N\in\N$. Thus
\begin{align*}
\Pb\left[\sup_{t\in[0,T]}\norms{\upsilon^N(t)}\geq\varepsilon\right]&\leq\Pb\left[a^N u^N\geq \varepsilon\right]= \Pb[a^N u^N\geq \varepsilon, u^N\leq M]+
\Pb[a^N u^N\geq \varepsilon, u^N> M]\\
&\leq\Pb\left[a^N\geq \varepsilon/M\right]+\Pb\left[u^N> M\right]\leq\Pb\left[a^N\geq \varepsilon/M\right]+\delta.
\end{align*}
Given that the $\delta$ was arbitrary, the result then follows from the fact that $S^{(N)}$ converges in probability to $S$ (Theorem \ref{thm:weak conv of Skn}).
\end{proof}}
\subsection{Analysis of Noise}\label{sec:noisexkN}
The proof of Lemma \ref{lem:noise1} is based on \cite[Lemma 8.9]{KOS16}. For the reader's convenience, we restate  \cite[Lemma 8.9]{KOS16} below as   Lemma \ref{lem:finidimdistr+tightness}. In order to state such a lemma let us introduce the following notation and definitions. Let $k_N:[0,T] \ra \mathbb{Z}_+$ be a sequence of nondecreasing, right continuous functions indexed by $N$, with $k_N(0)=0$ and $k_N(T)\geq 1$. 
Let $\h$ be any Hilbert space and $\{X\kkn, \mathcal{F}\kkn\}_{0\leq k \leq k_N(T)}$ be a $\h$-valued martingale 
difference array (MDA), i.e. a double sequence of random variables such that $\EE[X\kkn\vert\mathcal{F}_{k-1}^N ]=0$, $\EE[\|{ X\kkn}\|^2\vert\mathcal{F}_{k-1}^N ]< \infty$ almost surely and sigma-algebras $\mathcal{F}^{k-1, N} \subseteq \mathcal{F}\kkn$. Consider the process $\mathcal{X}^N(t)$ defined by
$$
\mathcal{X}^N(t):=\sum_{k=1}^{k_N(t)}X^{k,N} \,,
$$
if $k_N(t)\geq 1$ and $k_N(t) > \lim_{v\ra 0+} k_N(t-v)$ and by linear interpolation otherwise.   With this set up we recall the following result.

\begin{lemma}[Lemma 8.9 in \cite{KOS16}]\label{lem:finidimdistr+tightness}
 Let {$D:\h \ra \h$} be a self-adjoint positive definite trace class operator on {$(\h, \norm{\cdot})$}. Suppose {the following limits hold in probability}
\begin{description}
\item[i)] there exists a continuous and positive function $f:[0,T]\to\R_+$ such that
$$\lim_{N\ra \infty} \sum_{k=1}^{k_N(T)} \EE({\norm{X\kkn}}^2\vert \mathcal{F}_{k-1}^N)= \tr_{\h}(D) \int_0^T f(t) dt \, ; $$ 

\item[ii)] {if $\{{\phi}_j\}_{j\in\N}$ is an orthonormal basis of $\h$ then 
$$
\lim_{N\ra \infty} \sum_{k=1}^{k_N(T)} 
\EE(\langle X\kkn,{\phi}_j  \rangle \langle X\kkn,{\phi}_i  \rangle\vert \mathcal{F}_{k-1}^N)=0\, \quad \mbox{for all }\,\, i\neq j\, ;
$$}
\item[iii)]for every fixed $\epsilon>0$,
$$
\lim_{N \ra \infty} \sum_{k=1}^{k_N(T)}
\EE({\norm{X\kkn}}^2 {\bf 1}_{\left\{{\norm{X\kkn}}^2\geq \epsilon \right\}}
\vert \mathcal{F}_{k-1}^N )=0,  \qquad\mbox{in probability},
$$
\end{description}
where $\mathbf{1}_A$ denotes the indicator function of the set $A$. Then the sequence $\mathcal{X}^N$ converges weakly in {$C([0,T]; \hs)$} to the stochastic integral {$t\mapsto\int_0^t \sqrt{f(v)} dW_v$},  where $W_t$ is a {$\h$-valued $D$-Brownian motion}.
\end{lemma}
\begin{proof}[Lemma \ref{lem:noise1}] We apply Lemma \ref{lem:finidimdistr+tightness} in the Hilbert space $\hs$,   with $k_N(t)=[t\sqrt{N}]$,  $X^{k,N}=L\kkn/{N}^{1/4}$  ($L\kkn$ is defined in \eqref{Mkn})
 and  $\mathcal{F}_k^N$  the sigma-algebra generated by {$\{\gamma^{h,N}, \xi^{h,N}, \, 0\leq h\leq k\}$}  to study the sequence $\eta^N(t)$, defined in (\ref{etaN}). We now check that the three conditions of Lemma \ref{lem:finidimdistr+tightness} hold in the present case.  
 \begin{description}
{\item[i)] Note that by the definition of $L^{k,N}$, $\EE[L^{k,N}\vert\mathcal{F}_{k-1}^N]=\Ebc[L^{k,N}]$ almost surely.}{We need to show that the limit
\be\label{m1goestoint}
\lim_{N\to\infty}\frac{1}{\sqrt{N}}\sum_{k=0}^{[T\sqrt{N}]} \EE_k \nors{L\kkn}^2 = 2 \, \tr_{\h^s}(\C_s) \int_0^T h\el(S(u))du \, ,
\ee
holds in probability. By \eqref{LK},      
\begin{align*}
\frac{1}{\sqrt{N}} \EE_k \nors{L\kkn}^2  
& = \EE_k \nors{x^{k+1,N}-x^{k,N}}^2 
- \nors{\EE_k\left( x^{k+1,N}-x^{k,N}\right)}^2.
\end{align*}
From the above, if we prove
\be\label{n1}
\exo\sum_{k=0}^{[T\sqrt{N}]}\nors{\EE_k\left( x^{k+1,N}-x^{k,N}\right)}^2 \ra 0 \quad \mbox{as } N\ra \infty,
\ee
and that 
\be\label{n2}
\lim_{N\to\infty}\sum_{k=0}^{[T\sqrt{N}]}
\EE_k \nors{x^{k+1,N}-x^{k,N}}^2 = 2 \, \tr_{\h^s}(\C_s) \int_0^T h\el(S(u))du, \quad \mbox{in probability},
\ee
then \eqref{m1goestoint} follows. We start by proving \eqref{n1}:
\begin{align*}
\nors{\EE_k\left( x^{k+1,N}-x^{k,N}\right)}^2 & \stackrel{\eqref{chainxcomponents}}{\les} \nors{x\kkn +\C_N \nabla\Psi^N(x\kkn)}^2 +\frac{1}{\sqrt{N}} \nors{\EE_k \left( \gamma\kkn (\C_N)^{1/2}\xi\kkn  \right)}^2\\
&\lesssim \frac{1}{N} \left( 1+ \nors{x^{k,N}}^2\right),
\end{align*}
where the last inequality follows from \eqref{eq:lipz} and \eqref{normeps}.The above and \eqref{lem:boundedness2} prove \eqref{n1}. We now come to \eqref{n2}:
\begin{align*}
&\lv\sum_{k=0}^{[T\sqrt{N}]}
\EE_k \nors{x^{k+1,N}-x^{k,N}}^2-2 \, \tr_{\h^s}(\C_s) \int_0^T
 h\el(S(u))du \rv \\
&\stackrel{\eqref{chainxcomponents}}{\les}
\frac{1}{N}\sum_{k=0}^{[T\sqrt{N}]}
\EE_k \nors{x\kkn +\C_N \nabla\Psi^N(x\kkn)}^2\\
&+ \frac{1}{N^{3/4}}\sum_{k=0}^{[T\sqrt{N}]}
\EE_k \lv  \langle x\kkn +\C_N \nabla\Psi^N(x\kkn), \C_N^{1/2}\xi\kkn \rangle_s\rv\\
&+ \lv\frac{2\ell}{\sqrt{N}}\sum_{k=0}^{[T\sqrt{N}]}\EE_k \nors{\gamma\kkn \C_N^{1/2}\xi\kkn}^2 -2 \, \tr_{\h^s}(\C_s) \int_0^T
 h\el(S(u))du \rv.\\
 \end{align*}
 The first two addends tend to zero in $L_1$ as $N$ tends to infinity due to \eqref{eq:lipz}, \eqref{c1/2xi} and Lemma \ref{lem:moments}.  As for the third addend, we decompose it as follows
\begin{align}
&\lv\frac{2\ell}{\sqrt{N}}\sum_{k=0}^{[T\sqrt{N}]}\EE_k \nors{\gamma\kkn \C_N^{1/2}\xi\kkn}^2 -2 \, \tr_{\h^s}(\C_s) \int_0^T
 h\el(S(u))du \rv \nonumber\\
 &\stackrel{\eqref{def:hl}, \eqref{eps}}{\les} \lv \frac{\ell}{\sqrt{N}}\sum_{k=0}^{[T\sqrt{N}]}\EE_k \nors{\varepsilon^{k,N}}^2
 - \frac{\ell}{\sqrt{N}}\sum_{k=0}^{[T\sqrt{N}]}\tr_{\hs}(\C_s)\al(S^{k,N})\rv \nonumber\\
 &+ \lv\frac{1}{\sqrt{N}}\sum_{k=0}^{[T\sqrt{N}]}\tr_{\hs}(\C_s)\hl(S^{k,N})-    \tr_{\h^s}(\C_s) \int_0^T
 h\el(S(u))du\rv.  \label{1000}
\end{align}
The first addend in the above tends to zero in $L_1$ due to Lemma \ref{lem:AI}. As for the term in \eqref{1000}, 
we use the identity
$$\int_0^T\hl(\bar{S}^{(N)}(u))du =\left(T-\frac{[T\sqrt{N}]}{\sqrt{N}}\right)\hl(S^{[T\sqrt{N}],N}) +\frac{1}{\sqrt{N}}\sum_{k=0}^{[T\sqrt{N}]}\hl(S^{k,N}),$$
to further split it, obtaining:
\begin{align}
\eqref{1000}&\les 
\lv \int_0^T
\hl(\bar{S}^{(N)}(u))-
\hl(S^{(N)}(u))du \rv\label{10001}\\
&+\lv\int_0^T
\hl(S^{(N)}(u))-
\hl(S(u))du \rv\label{10002}\\
&+\left(T-\frac{[T\sqrt{N}]}{\sqrt{N}}\right)\hl(S^{[T\sqrt{N}],N}).\label{10003}
\end{align}
 Convergence (in $L_1$) of \eqref{10001} to zero  follows with the same calculations leading to \eqref{eskp1-sk},  the global Lipschitz property of $\hl$, and Lemma \ref{lem:moments}. The  addend in \eqref{10002}  tends to zero in probability since $S^{(N)}$ tends to $S$ in probability in $C([0,T];\R)$ (Theorem \ref{thm:weak conv of Skn}) and the third addend is clearly small. The limit \eqref{n2} then follows.}
\item[ii)] {Condition \textbf{ii)} of Lemma \ref{lem:finidimdistr+tightness} can be shown to hold with similar calculations, so we will not show
the details.}
\item[iii)] Using \eqref{eqlem:bla3} ,  the last bound follows a calculation completely analogous to the one  in \cite[Section 8.2]{KOS16} so we don't repeat details here.  


%
%
%
%

%
%
\end{description}
\end{proof}
%
%
\noindent
\appendix
\section{Proofs of the Results in Sections \ref{sec:2}}\label{misc}
%
%
\begin{proof}[Proof of Lemma \ref{lem:lipschitz+taylor}] The bounds \eqref{eq:lipz2}  are a consequence of \eqref{eq:C2}. We show how to obtain the second bound in \eqref{eq:lipz2}:
\begin{align*}
\norms{\C\nabla\Psi(x)-\C\nabla\Psi(y)}^2 & = \sum_{j=1}^{\infty}
\lambda_j^4 j^{2s} \left[\left( \nabla\Psi(x)-\nabla\Psi(y)\right)_j\right]^2\\
&=\sum_{j=1}^{\infty}
(\lambda_j j^{s})^4 j^{-2s} \left[\left( \nabla\Psi(x)-\nabla\Psi(y)\right)_j\right]^2\\
&\les \|\nabla\Psi(x)-\nabla\Psi(y)\|_{-s}^2 
\stackrel{\eqref{eq:C2}}{\les} \|x-y\|_s^2,
\end{align*}
where in the above we have used \eqref{bddseq} and  $\left( \nabla\Psi(x)-\nabla\Psi(y)\right)_j$ denotes the $j$-th component of the vector $\nabla\Psi(x)-\nabla\Psi(y)$. With analogous calculations one can obtain  the first bound in \eqref{eq:lipz2}. As for the second equation in \eqref{e.Flipshitz}:
\begin{align*}
\norms{F(z)} &\les \nors{z}+ \| \C \nabla \Psi(z)\|_{s} \stackrel{\eqref{eq:lipz2}} {\les} 1+ \nors{z} \,. 
\end{align*}
Similarly for the first bound in \eqref{e.Flipshitz}.  The proof of equation \eqref{eq:taylor} is standard, so we only sketch it: consider a line joining points $x$ and $y$, $\gamma(t)= x+t(y-x), t \in [0,1]$. Then
\begin{align*}
\Psi(\gamma(1))-\Psi(\gamma(0)) & =\Psi(y)-\Psi(x)\\
&= \int_0^1 dt \,\iprod{\nabla\Psi(\gamma(t))}{ y-x} \les 
\norms{y-x},
\end{align*}
 having  used \eqref{eq:C2} and \eqref{eq:B2} in the last inequality.
An analogous calculation to the above can be done for $\Psi^N$, after proving \eqref{eq:lin} below.
\end{proof}
\begin{proof}[Proof of Lemma \ref{lemma2.6}]
The bounds \eqref{eq:lin} and \eqref{eq:lipz} are just consequences of the definition of $\Psi^N$ and $\nabla\Psi^N$ and the analogous properties of $\Psi$. For the sake of clarity we just spell out how to obtain \eqref{eq:lipz}:
\begin{align*}
\nors{\C_N \nabla \Psi^N(x)}^2&\stackrel{\eqref{gradpsiN}}{=} \nors{\C_N \PP^N\nabla \Psi(\PP^N(x))}^2  = \sum_{j=1}^N j^{2s}\lambda_j^4\left[\nabla\Psi (\PP^N(x)) \right]_j^2\\
&\leq \sum_{j=1}^{\infty} j^{2s}\lambda_j^4\left[\nabla\Psi (\PP^N(x)) \right]_j^2\leq \nors{\C \nabla \Psi(\PP^N(x))}^2\stackrel{\eqref{eq:lipz2}}{\les}1\,.
\end{align*}
 {As for \eqref{eq:B4}, using \eqref{bddseq}:
\begin{align*}
\|\C_N \nabla\Psi^N(x)\|_{\C_N}^2 & =\sum_{j=1}^{N}
\lambda_j^2 \left[ \left(\nabla\Psi^N(x)\right)_j\right]^2
 \les \sum_{j=1}^{\infty}  j^{-2s} \left[ \left(\nabla\Psi^N(x)\right)_j\right]^2 = \|\nabla\Psi^N(x)\|_{-s}^2\les 1.
\end{align*}}
\end{proof}
\section{Proofs of Lemmas \ref{lem:epsilon} and \ref{lem:AI}}\label{corrproofs}
In view of the proof of Lemma \ref{lem:epsilon} and Lemma \ref{lem:AI},  let us decompose $Q^N(x^{k,N}, \xi^{k,N})$ into a term that depends on $\xi_j^{k,N}$ (the $j$-th component of $\xi^{k,N})$, $Q^N_j$,  and a term that is independent of $\xi_j$, $Q_{j,\perp}^N$:
$$
Q^N(x,\xi)=Q^N_j+Q_{j,\perp}^N, 
$$
where
\begin{align}
Q^N_j(x^{k,N},\xi^{k,N})&:=\left(\frac{\ell^{5/2}}{\sqrt{2}N^{5/4}}-\frac{\ell^{3/2}}{\sqrt{2}N^{3/4}}\right)\frac{x^{k,N}_j\xi^{k,N}_j}{\lambda_j}+\frac{\ell^{5/2}}{\sqrt{2}N^{5/4}}\lambda_j\xi^{k,N}_j(\nabla \Psi^N(x^{k,N}))_j \nonumber \\
&-\frac{\ell^2}{2N}(\xi_j^{k,N})^2
+I_2^N(x^{k,N},y^{k,N})+I_3^N(x^{k,N},y^{k,N}) \,.\label{Qj}
\end{align}
We recall that $I_2^N$ and $I_3^N$ have been defined in Section \ref{sec7}. 
Therefore, using \eqref{QNdecomp}, 
\be\label{Qjperp}
Q_{j,\perp}^N= Q^N-Q^N_j=I_1^N+\tilde{Q}_j^N,
\ee
having set
\be\label{Qjtilde}
\tilde{Q}_j^N:=-\left(\frac{\ell^{5/2}}{\sqrt{2}N^{5/4}}-\frac{\ell^{3/2}}{\sqrt{2}N^{3/4}}\right)\frac{x^{k,N}_j
\xi^{k,N}_j}{\lambda_j}
-\frac{\ell^{5/2}}{\sqrt{2}N^{5/4}}\lambda_j\xi^{k,N}_j(\nabla \Psi^N(x^{k,N}))_j
+\frac{\ell^2}{2N}(\xi_j^{k,N})^2. 
\ee

\begin{proof}[Proof of Lemma \ref{lem:epsilon}]\eqref{circle} is a consequence of the definition \eqref{eps} and the estimate \eqref{normeps}.
Thus, all we have to do is establish the latter. Recalling that $\{\hat{\phi}_j\}_{j\in\N}:= \{j^{-s}\phi_j\}_{j\in\N}$ is an orthonormal basis for $\hs$, we act as  in the proof of  \cite[Lemma 4.7]{MR3024970} and obtain
$$
\lv \iprods{{\Ebc\varepsilon\kkn}}{\hat{\phi}_j}\rv^2\lesssim j^{2s}\lambda_j^2\Ebc\left[
{Q^N_j(x^{k,N},\xi^{k,N})}\right]^2$$
where $Q^N_j$ has been defined in \eqref{Qj}.
Thus
\begin{align*}
\lv \iprods{{\Ebc\varepsilon\kkn}}{\hat{\phi}_j}\rv^2\lesssim & j^{2s}\lambda_j^2\left(N^{-3/2}{(x^{k,N}_j)^2\EE_k \xi_j^{2}}\lambda_j^2+N^{-5/2}\lambda_j^2\Ebc\left[{\xi_j^2(\nabla\Psi^N(x^{k,N}))_j^2}\right]\right)\\
&+ j^{2s}\lambda_j^2\EE_k(\lv I_2^N\rv^2+\lv I_3^N\rv^2)+ \frac{j^{2s}\lambda_j^2}{N^2}\\
&\lesssim N^{-3/2}\Ebc{(j^sx^{k,N}_j)^2} 
+N^{-5/2}j^{-2s}(\nabla\Psi^N(x^{k,N}))_j^2\\
&+j^{2s}\lambda_j^2N^{-2}+ j^{2s}\lambda_j^2\frac{1+\norms{x^{k,N}}^2}{\sqrt{N}},
\end{align*}
where the second inequality follows from the boundedness of the sequence $\{\lambda_j\}$, \eqref{lem:I2} and \eqref{lem:I3}. Summing over $j$ and applying \eqref{eq:lin} we obtain \eqref{normeps}. 
\end{proof}
{\begin{proof}[Proof of Lemma \ref{lem:AI}]  
By definition of $\varepsilon^{k,N}$, and because $\gamma^{k,N}=[\gamma^{k,N}]^2$ (as $\gamma^{k,N}$ can only take values 0 or 1)
$$\Ebc\norms{\varepsilon\kkn}^2=\sum_{j=1}^N j^{2s}\lambda_j^2\Ebc\left[\gamma^{k,N}\mmag{\xi^{k,N}_j}^2
\right] = \sum_{j=1}^N j^{2s}\lambda_j^2\Ebc\left[\left(1\wedge e^{Q^N(x^{k,N},y^{k,N})}\right)\mmag{\xi^{k,N}_j}^2\right].$$
Using the above,  the  Lipschitzianity of the function $s \mapsto1\wedge e^s$,  \eqref{Qjperp} and the independence of $Q_{j,\perp}^N$ and $\xi_j^{k,N}$,  we write
\begin{align}
\lv \Ebc\norms{\varepsilon\kkn}^2 - \tr(\C_s)\al(S^{k,N}) \rv & = 
\lv \EE_k \sum_{j=1}^N j^{2s}\lambda_j^2\left(1\wedge e^{Q^N}\right)
\lv \xi_j\rv^2  -\tr(\C_s)\al(S^{k,N}) \rv \nonumber\\
&\leq \lv \EE_k \sum_{j=1}^N j^{2s}\lambda_j^2\left(1\wedge e^{Q^N_{j,\perp}}\right)
\lv \xi_j\rv^2 - \tr(\C_s)\al(S^{k,N}) \rv \nonumber\\
&+\lv \EE_k \sum_{j=1}^N j^{2s}\lambda_j^2 \left[
\left(1\wedge e^{Q^N}\right)-\left(1\wedge e^{Q^N_{j,\perp}}\right) \right]
\lv \xi_j\rv^2  \rv \nonumber\\
&\les \lv  \sum_{j=1}^N j^{2s}\lambda_j^2 \EE_k\left(1\wedge e^{Q^N_{j,\perp}}\right) - \tr(\C_s)\al(S^{k,N}) \rv \label{mess1}\\
&+ \lv \EE_k \sum_{j=1}^N j^{2s}\lambda_j^2 \lv Q_j^N \rv
\lv \xi_j\rv^2 \rv \label{mess2}
\end{align}

We now proceed to bound the addends in  \eqref{mess1} and \eqref{mess2}, starting from the latter. Using \eqref{Qj} and \eqref{Qjtilde}, we write
\begin{align*}
\exo\EE_k \sum_{j=1}^N j^{2s}\lambda_j^2 \lv Q_j^N \rv
\lv \xi_j\rv^2&\leq 
\exo \EE_k \sum_{j=1}^N j^{2s}\lambda_j^2 \lv I_2^N \rv
\lv \xi_j\rv^2+ \exo \EE_k \sum_{j=1}^N j^{2s}\lambda_j^2 \lv I_3^N \rv
\lv \xi_j\rv^2\\
&+\exo\EE_k \sum_{j=1}^N j^{2s}\lambda_j^2 \lv \tilde{Q}_j^N \rv
\lv \xi_j\rv^2\\
&\les \exo \sum_{j=1}^N j^{2s}\lambda_j^2 (\EE_k\lv I_2^N \rv^2)^{1/2}
+ \exo \EE_k \sum_{j=1}^N j^{2s}\lambda_j^2 (\EE_k\lv I_3^N \rv^2)^{1/2} \\
&+ \exo \sum_{j=1}^N j^{2s}\lambda_j^2 \EE_k \left( \lv \tilde{Q}_j^N \rv
\lv \xi_j\rv^2\right).
\end{align*}
The  addends on the penultimate line  of the above  tend to zero thanks to Lemma \ref{lem:lem:I1}, \eqref{ljf} and Lemma \ref{lem:moments}. As for the last addend, using \eqref{Qjtilde}:

\begin{align}\sum_{j=1}^N j^{2s}\lambda_j^2 \EE_k \left[ \lv \tilde{Q}_j^N \rv
\lv \xi_j\rv^2\right]& \les  \frac{1}{N^{3/4}} \sum_{j=1}^N j^{2s}\lambda_j \lv x\kkn_j\rv \EE_k\lv \xi_j^{k,N} \rv^3\nonumber\\
&+\frac{1}{N^{5/4}} \sum_{j=1}^N j^{2s}\lambda_j^3 \lv (\C_N\nabla\Psi^N(x^{k,N}))_j\rv \EE_k\lv \xi_j^{k,N} \rv^3+ \frac{1}{N}\sum_{j=1}^Nj^{2s}\lambda_j^2\Ebc\lv \xi_j^{k,N} \rv^4 \nonumber\\
&\les \frac{1}{N^{3/4}}(1+ \nors{x\kkn}^2), \label{anotherlabel}
\end{align}
where the last inequality follows from \eqref{eq:lipz}, \eqref{ljf},  the boundedness of the  sequence $\{\lambda_j\}_{j\in\N}$ and by using the Young Inequality (more precisely, the so-called Young inequlity ``with $\epsilon$"), as follows:
$$
\lambda_j \lv x\kkn_j\rv \EE_k\lv \xi_j^{k,N} \rv^3 \leq  \lv x\kkn_j\rv^2 + \lambda_j^2 \left(\EE_k\lv \xi_j^{k,N} \rv^3\right)^2.
$$
This concludes the analysis of the term \eqref{mess2}. As for the term \eqref{mess1}, by definition of $\al$, equation \eqref{def:al}, 
\begin{align*}
\left(1\wedge e^{Q^{k,N}_{j,\perp}}\right)-\al(S^{k,N})&=\left(1\wedge e^{Q^{k,N}_{j,\perp}}\right)-\left(1\wedge e^{I^N_1(x^{k,N},y^{k,N})}\right)\\
& +\left(1\wedge e^{I^N_1(x^{k,N},y^{k,N})}\right)-\left(1 \wedge e^{\ell^2(S^{k,N}-1)/2}\right).
\end{align*}
Because $s\mapsto1\wedge e^s$ is globally Lipschitz, using  Lemma \ref{lem:lem:I1} and manipulations of the same type as in the above, we conclude that also \eqref{mess1} tends to zero as $N\ra \infty$. 
This concludes the proof. 
 \end{proof}}
%
%

%
\section{Uniform Bounds on the Moments of $S^{k,N}$ and $x^{k,N}$}\label{momproofs}
\begin{proof}[Proof of Lemma \ref{lem:moments}] To prove both bounds, we use a strategy analogous to the one used in \cite[Proof of Lemma 9]{Pillai2014}.    Let $\{A_k:k\in\mathbb{N}\}$ be any sequence of real numbers. Suppose that there exists a constant $C\geq0$ (independent of $k$) such that 

\begin{equation}\label{eq:recur}
A_{k+1}-A_k\leq \frac{C}{\sqrt{N}}\left(1+A_k\right).
\end{equation}
We start by showing that if the above holds then 
 $A_k\leq e^{CT}(A_0+CT)$, uniformly over $k=0,\dots,[T\sqrt{N}]$. Indeed, from \eqref{eq:recur}, 
$$A_k\leq \left(1+\frac{C}{\sqrt{N}}\right)^kA_0+\frac{C}{\sqrt{N}}\sum_{j=0}^{k-1}\left(1+\frac{C}{\sqrt{N}}\right)^j \leq \left(1+\frac{C}{\sqrt{N}}\right)^k\left(A_0+k\frac{C}{\sqrt{N}}\right).$$
Thus, for all $k=0,\dots,[T\sqrt{N}]$, 
%
$$A_k\leq \left(1+\frac{C}{\sqrt{N}}\right)^{[T\sqrt{N}]}(A_0+[T\sqrt{N}]\frac{C}{\sqrt{N}})\leq \left(1+\frac{C}{\sqrt{N}}\right)^{T\sqrt{N}}(A_0+CT).$$
Since $[0, \infty)\ni N\mapsto(1+C/\sqrt{N})^{\sqrt{N}}$ is  increasing, $$\left(1+\frac{C}{\sqrt{N}}\right)^{\sqrt{N}}\leq \left(1+\frac{C}{\left\lceil\sqrt{N}\right\rceil}\right)^{\left\lceil\sqrt{N}\right\rceil}\leq \sum_{j=0}^{\left\lceil\sqrt{N}\right\rceil}\frac{C^j}{j!}\leq e^C.$$
With this preliminary observation, we can now prove \eqref{eq:bounded2} and 
 \smallskip
\begin{description}
\item[i) Proof of \eqref{eq:bounded2}.]To prove \eqref{eq:bounded2} we only need to show that \eqref{eq:recur} holds (for some constant $C>0$ independent of $N$ and $k$) for the sequence $A_k=\exo{(S^{k,N})^q}$. 
By the definition of $S^{k,N}$, we have
$$S^{k+1,N} = S^{k,N}+\frac{\normcn{x^{k+1,N}-x^{k,N}}^2}{N} +\frac{2\iprodcn{x^{k+1,N}-x^{k,N}}{x^{k,N}}}{N}.$$
Therefore, 
\begin{align}\label{eq:multi}
&\exo{(S^{k+1,N})^q}- \exo(S^{k,N})^q \nonumber\\
&=\sum_{\substack{n+m+l=q \\(n,m,l)\neq (q,0,0)}}\!\!\! \exo \left[(S^{k,N})^n\left(\frac{\normcn{x^{k+1,N}-x^{k,N}}^2}{N}\right)^m
\left(\frac{2\iprodcn{x^{k+1,N}-x^{k,N}}{x^{k,N}}}{N}\right)^l\right].
\end{align}
Thus, to establish \eqref{eq:recur} it is enough to argue that each of the terms in the right-hand side of the above  is bounded by $(C/\sqrt{N})(1+\EE{(S^{k,N})^q})$. To this end, set  
\begin{align*}
J^{k,N}&:= \exo \left[ {(S^{k,N})^n\left(\frac{\normcn{x^{k+1,N}-x^{k,N}}^2}{N}\right)^m
\left(\frac{2\iprodcn{x^{k+1,N}-x^{k,N}}{x^{k,N}}}{N}\right)^l}
\right]\\
&= \exo  \Ebc\left[ {(S^{k,N})^n\left(\frac{\normcn{x^{k+1,N}-x^{k,N}}^2}{N}\right)^m\left(\frac{2\iprodcn{x^{k+1,N}-x^{k,N}}{x^{k,N}}}{N}\right)^l}
\right]. 
\end{align*}
By the Cauchy-Schwartz inequality for the scalar product $\iprodcn{\cdot}{\cdot}$,
\begin{align*}
\frac{\iprodcn{x^{k+1,N}-x^{k,N}}{x^{k,N}}^l}{N^l} &\leq \frac{\normcn{x^{k,N}}^l\normcn{x^{k+1,N}-x^{k,N}}^l}{N^l}\\
& =(S^{k,N})^{l/2}\frac{\normcn{x^{k+1,N}-x^{k,N}}^l}{N^{l/2}},
\end{align*}
which gives 
\begin{align*}   J_k^N\lesssim (S^{k,N})^{n+l/2}\frac{\Ebc{\normcn{x^{k+1,N}-x^{k,N}}^{2m+l}}}{N^{m+l/2}}.\end{align*}
Using the bound \eqref{eqlem:bla4} of  Lemma \ref{lem:bla},  we also  have 
$$ \EE_k\frac{{\normcn{x^{k+1,N}-x^{k,N}}^{2m+l}}}{N^{m+l/2}}\lesssim \frac{(S^{k,N})^{m+l/2}}{N^{m+l/2}}+\frac{1}{N^{(m+l/2)/2}}.$$
Putting all of the above together (and using Young's inequality)  we obtain
$$
J_k^N \les \frac{\exo(S^{k,N})^q}{N^{m+l/2}}+ \frac{1}{N^{m+l/2}}. 
$$
 Now observe that $(m+l/2)/2\geq 1/2$ except when  $(n,m,l)=(q,0,0)$ or  $(n,m,l)=(q-1,0,1)$. Therefore we have shown the desired bound for all the terms in the expansion (\ref{eq:multi}),  except the one with $(n,m,l)=(q-1,0,1)$.  To study the latter term, we recall that  $\gamma^{k,N}\in\{0,1\}$, and use  the definition of the chain (equations \eqref{eqn:proposal} and \eqref{chain-gamma}) to obtain
\begin{align*}
\mmag{\iprodcn{x^{k+1,N}-x^{k,N}}{x^{k,N}}} &\lesssim \delta\normcn{x^{k,N}}^2+
\delta\mmag{\iprodcn{\C_N\nabla\Psi^N(x^{k,N})}{x^{k,N}}}\\
 & + \sqrt{\delta}\mmag{\iprodcn{x^{k,N}}{(\C_N)^{1/2}\xi^{k,N}}}.
\end{align*}
Combining \eqref{eq:B4} with the Cauchy-Schwartz inequality we have 
$$\delta\mmag{\iprodcn{\C_N\nabla\Psi^N(x^{k,N})}{x^{k,N}}}\lesssim N^{-1/2}(1+\norms{x^{k,N}}^2)\lesssim N^{-1/2}+N^{1/2}S^{k,N},$$
where in the last inequality we used the following observation 
$$\norms{x^{k,N}}^2=\sum_{j=1}^\infty(x^{k,N})_j^2j^{2s}=\sum_{j=1}^\infty\frac{(x^{k,N})_j^2}{\lambda_j^2}(\lambda_j^2j^{2s})\lesssim \sum_{j=1}^\infty\frac{(x^{k,N})_j^2}{\lambda_j^2}=NS^{k,N}.$$
Recalling that $\iprodcn{x^{k,N}}{(\C_N)^{1/2}\xi^{k,N}}$, conditioned on $x^{k,N}$, is a linear combination of zero-mean Gaussian random variables, we have
\begin{align*}
\Ebc{ \sqrt{\delta}\mmag{\iprodcn{x^{k,N}}{(\C_N)^{1/2}\xi^{k,N}}}} &\lesssim 1+N^{-1/2}\Ebc{\mmag{\iprodcn{x^{k,N}}{(\C_N)^{1/2}\xi^{k,N}}}^2}\\
& \lesssim 1+\sqrt{N}S^{k,N}.
\end{align*}
Putting the above together and taking expectations we can then conclude 
\begin{align*}
\Ebb{\frac{(S^{k,N})^{q-1}\iprodcn{x^{k+1,N}-x^{k,N}}{x^{k,N}}}{N}} & \lesssim \frac{\Ebb{(S^{k,N})^{q-1}}}{N}+\frac{\Ebb{(S^{k,N})^q}}{\sqrt{N}}\\
&\lesssim(1/\sqrt{N})(1+\Ebb{(S^{k,N})^q}),
\end{align*}
and \eqref{eq:bounded2} follows. 
\item[ii) Proof of \eqref{lem:boundedness2}.] This is very similar to the proof of \eqref{eq:bounded2}, so we only sketch it. Just as before,  it is enough to establish the following bound 
$$\Ebb{\norms{x^{k,N}}^{2n}\norms{x^{k+1,N}-x^{k,N}}^{2m}\iprods{x^{k+1,N}-x^{k,N}}{x^{k,N}}^{l}}\lesssim\frac{1}{\sqrt{N}}(1+\Ebb{\norms{x^{k,N}}^{2q}})$$
for each $(n,m,l)$ such that $n+m+l=q$ with the exception of the triple $(n,m,l)=(q,0,0)$. Applying the Cauchy-Schwartz inequality for $\iprods{\cdot}{\cdot}$ we have %
$$\iprods{x^{k+1,N}-x^{k,N}}{x^{k,N}}^{l}\leq \norms{x^{k,N}}^l\norms{x^{k+1,N}-x^{k,N}}^l.$$
Thus,  Lemma \ref{lem:bla} implies 
\begin{align*}\Ebc{\norms{x^{k,N}}^{2n}\norms{x^{k+1,N}-x^{k,N}}^{2m}\iprods{x^{k+1,N}-x^{k,N}}{x^{k,N}}^{l}}&\leq \norms{x^{k,N}}^{2n+l}\Ebc{\norms{x^{k+1,N}-x^{k,N}}^{2m+l}} \\ 
&\lesssim \frac{\norms{x^{k,N}}^{2n+l}(1+\norms{x^{k,N}}^{2m+l})}{N^{(m+l/2)/2}}.\end{align*}
The above gives us the desired bound for all $(n,m,l)$ except for $(n,m,l)=(q-1,0,1)$. Like before, to study the latter case we observe
\begin{align*}\iprods{x^{k+1,N}-x^{k,N}}{x^{k,N}}&=\gamma^{k,N}(-\frac{1}{\sqrt{N}}(\norms{x^{k,N}}^2+\iprods{\C_N\nabla\Psi^N(x^{k,N})}{x^{k,N}})
\\ &+\frac{\sqrt{2}}{N^{1/4}}\iprods{(\C_N)^{1/2}\xi^{k,N}}{x^{k,N}})\\
&\lesssim \frac{1}{\sqrt{N}}(1+\norms{x^{k,N}}^2)+\frac{1}{N^{1/4}}\gamma^{k,N}\iprods{(C_N)^{1/2}\xi^{k,N}}{x^{k,N}}\\
& \lesssim\frac{1}{\sqrt{N}}(1+\norms{x^{k,N}}^2),\end{align*}
where penultimate inequality follows from the Cauchy-Schwartz inequality, \eqref{eq:lipz}, and the fact that $\gamma^{k,N}\in\{0,1\}$, and the last inequality follows from Lemma \ref{lem:epsilon}.  This concludes the proof. 
\end{description}
\end{proof}
\begin{remark}\label{Remmomchain}\textup{ In \cite{MR3024970} the authors derived the diffusion limit for the chain under weaker assumptions on the potential $\Psi$ than those we use in this paper. Essentially, they assume that $\Psi$ is quadratically bounded, while we assume that it is linearly bounded. If $\Psi$ was quadratically bounded the proof of Lemma \ref{lem:epsilon} would become considerably more involved. We observe explicitly that the statement of  Lemma \ref{lem:epsilon} is of paramount importance in order to establish the uniform bound on the moments of the chain $x^{k}$ contained in  Lemma \ref{lem:moments}.  In \cite{MR3024970} obtaining such bounds is not an issue, since the authors  study the chain  in its stationary regime. In other words, in \cite{MR3024970} the law of $x^{k,N}$ is independent of $k$, and thus the uniform bounds on the moments of $x^{k,N}$ and $S^{k,N}$ are automatically true for target measures of the form considered there (see also the first bullet point of  Remark \ref{rem:scal}). }
\hf
\end{remark}
\bigskip

\noindent
{\bf Acknowledgments} A.M. Stuart acknowledges support from AMS, DARPA, EPSRC, ONR. J. Kuntz gratefully acknowledges support from the BBSRC in the form of the Ph.D
studentship BB/F017510/1. M. Ottobre and J. Kuntz gratefully acknowledge financial support from the Edinburgh Mathematical Society. 
%
%
%
\bibliographystyle{plain}     
\bibliography{difflim}
\end{document}